\documentclass[a4paper,11pt]{article}

\usepackage[utf8]{inputenc}
\DeclareUnicodeCharacter{00A0}{~}

\usepackage[english]{babel}

\usepackage{a4wide}
\usepackage{amssymb,amsmath,amscd,amsfonts,amsthm,bbm,mathrsfs,enumerate}
\usepackage[colorlinks=true, linkcolor=blue, citecolor=blue]{hyperref}
\usepackage{graphicx}
\usepackage{tikz,pgfplots}
\usepackage{color}
\usepackage{csquotes}
\usepackage{authblk}
\usepackage{comment}
\usepackage[numbers,sort&compress]{natbib}

\DeclareMathOperator{\rank}{rank}
\usepackage{shortcuts}

\newtheorem{theorem}{Theorem}[section]
\newtheorem{lemma}[theorem]{Lemma}
\newtheorem{proposition}[theorem]{Proposition}
\newtheorem{definition}[theorem]{Definition}
\newtheorem{example}[theorem]{Example}
\newtheorem{remark}[theorem]{Remark}
\newtheorem{assumption}{Assumption}
\newtheorem{corollary}[theorem]{Corollary}

\providecommand{\keywords}[1]{\textbf{\textbf{Keywords. }} #1}

\pgfplotsset{compat=1.18}
\DeclareMathOperator{\Jac}{Jac}

\title{The gradient's limit of a definable family of functions admits a variational stratification}
\author{Sholom Schechtman\thanks{SAMOVAR, Télécom SudParis, Institut Polytechnique de Paris, 91120 Palaiseau, France}}

\date{\today}

\begin{document}

\maketitle

\begin{abstract}
  It is well-known that the convergence of a family of smooth functions does not imply the convergence of its gradients. In this work, we show that if the family is definable in an o-minimal structure (for instance semialgebraic, subanalytic, or any composition of the previous with exp, log), then the gradient's limit admits a variational stratification and, under a uniform Lipschitz continuity assumption, is a conservative set-valued field in the sense introduced by Bolte and Pauwels. Immediate implications of this result on convergence guarantees of smoothing methods are discussed. The result is established in a general form, where the functions in the original family might be non Lipschitz continuous, be vector-valued and the gradients are replaced by their Clarke Jacobians or an arbitrary mapping satisfying a definable variational stratification. In passing, we investigate stability properties of definable variational stratifications and smoothing methods that produce definable functions which might be of independent interest.
\end{abstract}


\section{Introduction}

In this work, given a family of real-valued functions $(f_a)_{a >0}$ that converges, when $a \rightarrow 0$, to some function $F: \bbR^d \rightarrow \bbR$, we are interested in the behavior of the limit of the corresponding family of (sub)-gradients $(\partial f_a)_{a >0}$. Formally, we are interested in properties satisfied by
\begin{equation}\label{eq:def_cons_intro}
  D_F(x) := \{v \in \bbR^d: \textrm{ there is }\textrm{$(x_n,v_n,a_n) \rightarrow (x,v,0)$}\textrm{ with $v_n \in \partial f_a(x_n)$} \}\, .
\end{equation}

While characterizing such limits is interesting from a purely theoretical perspective, $D_F$ naturally appears in the analysis of smoothing methods. In these, given a nonsmooth and perhaps not even locally Lipschitz continuous function $F$, the general construction, dating back to at least Mayne and Polak (\cite{mayne1984nondifferential}), goes as follows (see e.g. \cite{chen2012smoothing}). \emph{i)} First, construct $(f_a)_{a >0}$ a family of approximations of $F$ such that every $f_a$ is smooth. \emph{ii)} Second, for fixed $a_k, \varepsilon_k>0$, find $x_k \in \bbR^d$ an approximate stationary point: $\norm{\nabla f_{a_k}(x_k)} \leq \varepsilon_k$. \emph{iii)} Finally, decrease $a_k, \varepsilon_k$ and go back to step \emph{ii)}.
The interest of this procedure lies in the fact that we replace the original nonsmooth problem with a sequence of smooth optimization problems, for which there exists an abundance of algorithms with theoretical convergence guarantees (\cite{NoceWrig06}).

In the case of a smoothing method, and assuming that $(x_k)$ has an accumulation point $x^*$, we immediately obtain that $x^*$ is $D_F$-critical: $0 \in D_F(x^*)$. Thus, ideally, one would wish $D_F$ to be reduced to one of the common first-order operators: to the gradient if $F$ is differentiable, to the convex subgradient if it is convex or to the Frechet or Clarke subgradient if $F$ is merely continuous. Famously, Attouch in \cite{attouch2006convergence} has shown that when every function in $(f_a)_{a >0}$ is convex, $D_F$ is indeed the convex subgradient. This result was extended in numerous ways, example given: for Banach spaces (\cite{attouch1993convergence}), weakly-convex functions (\cite{poliquin1992extension,levy1995partial}), and equi-lower semidifferentiable functions (\cite{zolezzi1985continuity,zolezzi1994convergence, czarnecki2006approximation}).

Regrettably, the following simple example shows that there is no hope to state an equivalent result in full generality. Let $f_a(x):= a \sin(x/a)$, then $f_a'(x) = \sin(x/a)$ and for any $x \neq 0$, $D_F(x) = [-1, 1]$, which is obviously larger than any (sub)-gradient of $F \equiv 0$. While unfortunate, this counter-example shows that the convergence's failure is due to a highly oscillatory behavior, of $f(x,a) = f_a(x)$, as $a \rightarrow 0$. Thus, naturally, one might wonder what could happen if we restrict $f$, as a function of \emph{both $x$ and $a$}, to a class of functions where such an oscillation phenomena does not occur.

Fortunately, the nonsmooth optimization literature (\cite{bolte2007clarke,iof08,dav-dru-kak-lee-19}) have already established a setting, where precisely this pathological behavior is ruled out: the class of functions \emph{definable in an o-minimal structure} (\cite{van1998tame,cos02}). The class of such functions is large. It includes any semialgebraic function, any analytic function restricted to a semialgebraic compact, the exponential and the logarithm. Moreover, definability is stable by most of common operators such as $\{+, -, \times, \circ, \circ^{-1}, \sup, \inf\}$, explaining their ubiquity in optimization.

Definable functions may be nonsmooth, nevertheless, their differentiability properties are relatively well-understood. For instance, it is well-known that the domain of a definable function $F$ can be partitioned (or more precisely \emph{stratified}) into manifolds (or \emph{strata}) such that restricted to each element of the partition the function is smooth (see \cite{van1998tame}). Furthermore, in the seminal work \cite{bolte2007clarke}, it was established that the pair $(F, \partial F)$, with $\partial F$ denoting the Clarke subgradient, admits the so-called \emph{variational stratification}: the partition can be chosen in a way ensuring that the projection of the Clarke subgradient onto the tangent plane of the corresponding manifold (on which the objective is smooth) is simply the Riemannian gradient of the smooth restriction.

This geometric description turned out to be a fruitful point of view in recent advances in nonsmooth optimization. It is, for instance, a central tool for establishing convergence of the (Clarke) subgradient method  (\cite{dav-dru-kak-lee-19}). Importantly, it is also closely related to the recently introduced by Bolte and Pauwels (\cite{bolte2021conservative}) notion of a \emph{conservative set-valued field}. The latter, is a ``differential-akin'' object, which, roughly speaking, acts as a directional derivative along any smooth curve. For definable functions, examples of conservative fields are the Clarke subgradient, but also the output of automatic differentiation of the loss function of a neural network through the application of the backpropagation algorithm. In fact, as shown in \cite{lewis2021structure, pauwels2023conservative,davis2022conservative}, in the definable setting, under mild regularity properties (such as being locally bounded and non empty-valued), any mapping that satisfies a variational stratification \emph{is} a conservative set-valued field.

At this point, we are ready to state our main result. 
\begin{theorem}[Particular case of Theorem~\ref{thm:main}]
  Let the function $f: x,a \mapsto f_a(x)$ be definable and, for each $a>0$, $f_a$ be continuous. Let $\partial f_a$ in~\eqref{eq:def_cons_intro} denote the Clarke subgradient of $f_a$ and assume that for all $x$, $\lim_{(x',a)  \rightarrow (x,0)}f_a(x') = F(x)$. Then, the pair $(F,D_F)$ admits a variational stratification.\footnote{As alluded before, the pair $(F,D_F)$ admits a $C^p$ variational stratification, for $p\geq 1$, if there is a finite partition of $\bbR^d$ into manifolds $(\cX_i)$ such that \emph{i)} $F$, restricted to any $\cX_i$, is $C^p$ \emph{ii)} and for any $x \in \cX_i$, $D_F(x) \subset \nabla_{\cX_i} F(x) + \cN_{\cX_i}(x)$, with $\nabla_{\cX_i} F(x)$ denoting the Riemannian gradient of $F_{|\cX_i}$ and $\cN_{\cX_i}(x)$ the normal plane of $\cX_i$ at $x$ (see Section~\ref{sec:var_strat}).}
\end{theorem}
Let us comment on some immediate implications of this result. First, the presented theorem has a very general form where, except definability, we do not impose any regularity assumptions on $F$ and $D_F$. This allows us to treat the important case where the function $F$ is not locally Lipschitz continuous, since, formally, for such $F$ the concept of a conservative set-valued field is not defined. However, as soon as $F$ is locally Lipschitz and $D_F$ satisfies the regularity properties of a conservative set-valued field (that is, it is locally bounded, closed and non empty-valued), we immediately obtain that $D_F$ is a set-valued conservative field in the sense of \cite{bolte2021conservative}. Second, $D_F$ might contain elements that are not (sub)-gradients of $F$. Nevertheless, from the variational stratification property of $D_F$ we immediately obtain that $D_F(x) \subset\partial F(x) =  \{ \nabla F(x)\}$ on an open dense set. In the vocabulary of smoothing methods this means that the \emph{gradient consistency} property holds almost everywhere (see \cite{chen2012smoothing}). Third, even for points where $D_F(x) \neq \partial F(x)$, the variational stratification gives a pleasant geometric description of $D_F$. Finally, due to the remarkable stability of definable functions, as soon as $F$ is definable, a large class of smoothing methods, such as the smoothings of finite max-functions proposed in \cite{chen2012smoothing}, produce a family $(f_a)_{a >0}$ which is indeed definable in the same as $F$ o-minimal structure. Thus, for such methods, the guarantee that $0 \in D_F(x^*)$ is a meaningful and necessary condition of optimality.

Let us also mention that our result is established in a more general setting. In fact, in the construction of $D_F$ in~\eqref{eq:def_cons_intro} the Clarke subgradients $\partial f_{a}$ can be replaced by a definable conservative field of $f_a$ or more generally by a mapping that satisfies the definable variation stratification property. Furthermore, similar conclusions hold when the functions are vector-valued, with $\partial f_{a} $ being replaced by the Clarke Jacobian or an arbitrary mapping that satisfies the definable variational stratification property. Finally, our main theorem is a consequence of more general results on the stability, by composition and graph's closure, of variational stratifications. We believe that these properties, established in the course of the paper, might be of independent interest.

\paragraph{Paper organization.} In Section~\ref{sec:prel} we introduce the necessary definitions of stratifications, o-minimal structures and conservative mappings. In Section~\ref{sec:var_strat} we investigate properties of definable variational stratifications. In particular, we establish that this notion is stable by (parametric) graph's closure. In Section~\ref{sec:def_limit} we present our main theorem. Implications of the latter for smoothing methods are discussed in Section~\ref{sec:smooth}. Here, we also present some examples of smoothing methods that produce definable functions. Section~\ref{sec:pf_prop_proj_con} and Appendix~\ref{sec:rem_pfs} are devoted to the remaining proofs.

\section{Preliminaries}\label{sec:prel}
\paragraph{Notations.} 
Given an open set $\cU \subset \bbR^d$, a $C^p$ map $g: \cU \rightarrow \bbR^m$ and $x \in \cU$, we denote $\Jac g(x) \in \bbR^{m\times d}$ the Jacobian of $g$ at $x$. Such a map is called a submersion if for every $x \in \cU$, $\Jac g(x)$ is surjective. Given $\cX \subset \bbR^d$, a map $g: \cX \rightarrow \bbR^m$ is said to be $C^p$ if for every $x \in \cX$, there $\cU$, an open neighborhood of $x$, and a $C^p$ map $\tilde{g} : \cU \rightarrow \bbR^m$ such that $\tilde{g}$ and $g$ agree on $\cX \cap \cU$. We will call $\tilde{g}$ a $C^p$ smooth representative of $g$ around $x$. For a map $g: \cX \rightarrow \bbR^m$, we define its graph as $\Graph g:= \{(x, g(x)) :x \in \cX \} \subset \bbR^{d+ m}$. For a linear map $J: \bbR^d \rightarrow \bbR^m$, we denote $\ker J$ its kernel.

We say that $D: \bbR^d \rightrightarrows \bbR^m$ is a set-valued map if for all $x \in \bbR^d$, $D(x) \subset \bbR^m$. For such map, $\Graph D := \{(x,y) : y \in D(x) \}$, and we say that $D$ is closed if $\Graph D$ is a closed subset of $\bbR^{d + m}$. It is said to be locally bounded if every $x \in \bbR^d$ admits a neighborhood $\cU \subset \bbR^d$ and $C>0$ such that $ \sup_{x \in U, v \in D(x)}\norm{v} \leq C$. For $A \subset \bbR^d$, we denote $\conv A$ its convex hull and $\overline{A}$ its closure.

\subsection{Submanifolds}

In this section we present some basic notions of differential geometry that will be used throughout the paper. 
We refer to \cite{lee2022manifolds,boumal2023intromanifolds} for a detailed introduction on these notions.
Given integers $p\geq 1$ and $d \geq k$, a set $\cM \subset \bbR^d$ is a $C^p$ submanifold of dimension $k$, if for every $x \in \cM$,
there is a neighborhood $\cU$ of $x$ and a $C^p$ submersion $g \colon \cU \rightarrow \bbR^{d-k}$, such that  $U \cap \cM = g^{-1}(\{0\})$. The tangent space at $x \in \cX$ is then $\cT_{\cM}(x) := \ker  \Jac g(x)$ and $\cN_{\cM}(x) := (\cT_{\cM}(x))^{\perp}$ is the normal space of $\cX$ at $x$.

For $\cM \subset \bbR^d$ a $C^p$ submanifold, we say that a function $f \colon \cM \rightarrow \bbR^m$ is $C^p$, if for every $x \in \cM$, there is $\tilde{f}$ a $C^p$ smooth representative of $f$ around $x$. In that case, the Riemannian\footnote{Here, the Riemannian structure on $\cM$ is implicitly induced from the usual Euclidian scalar product on $\bbR^d$.} gradient of $f$ at $x \in \cM$ is
$$
\nabla_{\cM} f(x) := P_{\cT_{\cM}(x)} \nabla \tilde{f}(x)\,,
$$
with $P_{\cT_{\cM}(x)}$ being the orthogonal projection onto $\cT_{\cM}(x)$ and $\tilde{f}$ is any smooth representative of $f$ around $x$.
Similarly, for a $C^p$ function $f = (f_1, \ldots, f_m) \colon  \cM \rightarrow \bbR^{m}$ and $x\in\cM$, we will denote 
\begin{equation*}
  \Jac_{\cM} f(x)  = [\nabla_{\cM} f_1(x), \ldots , \nabla_{\cM} f_m(x)]^{\top}  \in \bbR^{m \times d}
\end{equation*} 
its Riemannian Jacobian. Note that $\cN_{\cM}(x) \subset \ker \Jac_{\cM} f(x)$. Let us record a simple lemma.
\begin{lemma}[{\cite[Chapter 2, Problem (7)]{lee2022manifolds}}]\label{lm:tan_graph}
 Consider a $C^p$ submanifold $\cX \subset \bbR^d$ and a $C^p$ function $f: \cX \rightarrow \bbR^m$. The set $\Graph f \subset \bbR^{d+m}$ is a $C^p$ submanifold, with 
 \begin{equation*}
   \cT_{\Graph f}(x,f(x)) = \{(h, \Jac_{\cX} f(x) h) : h \in \cT_{\cX}(x)\}\, .
 \end{equation*}
\end{lemma}

Sometimes, it will be convenient to take the dual point of view: for a $C^p$ function $f: \cM \rightarrow \bbR^{m}$ and $x \in \cM$, we define the differential $\dif f(x) : \cT_{\cM}(x) \rightarrow \bbR^d$ as $\dif f(x) [h] = \Jac_{\cM} f(x) h$. We note, that if the image of $f$ is included in some manifold $\cF \subset \bbR^m$, then the image of $\dif f(x)$ is included in $\cT_{\cF}(f(x))$. Finally, the rank of $f$ at $x \in \cM$ is the rank of $\Jac_{\cM}f(x)$ (or the one of $\dif f(x)$).

\subsection{Stratifications}\label{sec:strat}

As we will see in the next section, most sets arising in optimization can be partitioned into manifolds. Such procedure is called a stratification. 
\begin{definition}
Let $A$ be a set in $\bbR^d$, a $C^p$ stratification of $A$ is a finite partition of $A$ into a family of \emph{strata} $\bbX = (\cX_i)$ such that each $\cX_i$ is a $C^p$ submanifold, and such that for any two distinct strata $\cX_i, \cX_j \in \bbX$,
\begin{equation*}
  \cX_i \cap \overline{\cX}_j \neq \emptyset \implies \cX_i \subset \overline{\cX}_j \backslash \cX_j \, .
\end{equation*}
Given a family $\bbA = \{ A_1, \dots, A_k\}$ of subsets of $\bbR^d$, we say that a stratification $\bbX$ is \emph{compatible} with $\bbA$, if for every $\cX \in \bbX$ and $A \in \bbA$ either $\cX \subset A$ or $\cX \cap A = \emptyset$.\end{definition}

Several conditions can be imposed on the way that different strata are glued together. We will be particularly interested in \emph{Whitney-(a) stratifications}. For $d \geq 1$ and two vector spaces $E_1, E_2 \subset \bbR^d$, define the angle between $E_1, E_2$ as
\begin{equation*}
  \angle(E_1, E_2) :=\begin{cases} \sup \{ \dist(u, E_2) : u \in E_1\, , \norm{u} = 1\} \quad &\textrm{ if $E_1 \neq \{ 0\}$}\, , \\
  0 \quad &\textrm{ otherwise} \, .
  \end{cases} 
\end{equation*}
Note that $E_1 \subset E_2$ is equivalent to the fact that $\angle(E_1, E_2) = 0$.
\begin{definition}[Whitney stratification of a set]\label{def:whitn_set}
  We say that a $C^p$ stratification $\bbX = (\cX_i)$ satisfies a \emph{Whitney-(a) property}, if for every pair of distinct strata $\cX_i, \cX_j$, for each $y \in \cX_i \cap \overline{\cX_j}$ and for each $\cX_j$-valued sequence $(x_n)$ such that $x_n \rightarrow y$,
  \begin{equation*}
    \angle(\cT_{\cX_i}(y), \cT_{\cX_j}(x_n)) \rightarrow 0\, .
  \end{equation*} 
 \end{definition}
 \begin{remark}\label{rmk:tancon_whitn}
  Note that if $x_n \rightarrow y \in \cX_i$, with $(x_n) \in (\cX_i)^{\bbN}$, then $\angle(\cT_{\cX_i}(y), \cT_{\cX_i}(x_n)) \rightarrow 0$. Therefore, Definition~\ref{def:whitn_set} means that for any sequence $x_n \rightarrow y \in \cX_i$, denoting $\cX_{i(n)}$ the stratum for which $x_n \in \cX_{i(n)}$, $\angle(\cT_{\cX_i}(y), \cT_{\cX_{i(n)}}(x_n)) \rightarrow 0$.
 \end{remark}
In this note, we will refer to $\bbX$ of the previous definition as a Whitney $C^p$ stratification. 

There are also several notions of stratifications of functions.
 \begin{definition}[Whitney stratification of a function {\cite[page 502]{van96}}]\label{def:func_whit}
  Consider a closed set $A \subset \bbR^d$, a set $B \subset \bbR^m$ and a function $f: A \rightarrow B$. A Whitney $C^p$ stratification of $f$ is a pair $(\bbX, \bbF)$, where $\bbX = (\cX_i)$ (respectively $\bbF = (\cF_i)$) is a Whitney $C^p$ stratification of $A$ (respectively of $B$) such that \emph{i)} for every $\cX \in \bbX$, $f_{|\cX}$ is $C^p$, \emph{ii)} for every $\cX \in\bbX$, $f(\cX) \in \bbF$ and \emph{iii)} $f_{|\cX}$ is of constant rank.
 \end{definition}
 
For a real-valued function, we can impose an additional condition on the stratification. This notion of \emph{Thom stratification} will be central in the proof of our main theorem.
 \begin{definition}[Thom ($a_f$) stratification]
  Consider $A \subset \bbR^d$, a \emph{continuous} function $f: A \rightarrow \bbR$ and $\bbX$ a Whitney $C^p$ stratification of $A$. $\bbX$ is said to satisfy the \emph{Thom $(a_f)$ condition} if \emph{i)} for every $\cX \in \bbX$, $f_{|\cX}$ is of constant rank, and \emph{ii)} for every pair of distinct strata $\cX_i, \cX_j \in \bbX$, for each $y \in \cX_i \cap \overline{\cX_j}$ and for each $\cX_j$-valued sequence $(x_n)$ such that $x_n \rightarrow y$,
  \begin{equation*}
  \angle( \ker \dif f_{|\cX_i}(y) , \dif f_{|\cX_j}(x_n) ) \rightarrow 0 \, .
  \end{equation*}
 \end{definition}

\subsection{Functions definable in an o-minimal structure} We collect here few elementary facts about functions and sets definable in an o-minimal structure. For more details, we refer to the monographs \cite{cos02,van1998tame,van96}. A nice review of their importance in optimization is \cite{iof08}.

The definition of an o-minimal structure is inspired by properties that are satisfied by semialgebraic sets.
\begin{definition}\label{def:o-min}
  We say that $\cO:=(\cO_n)$, where for each $n \in \bbN$, $\cO_n$ is a collection of sets in $\bbR^n$, is an o-minimal structure on the real field if the following holds. \emph{i)} If $Q: \bbR^n \rightarrow \bbR$ is a polynomial, then $\{x \in \bbR^n : Q(x) = 0 \} \in \cO_n$. \emph{ii)} For each $n \in \bbN$, $\cO_n$ is a boolean algebra: if $A, B \in \cO_n$, then $A \cup B, A \cap B$ and $A^c$ are in $\cO_n$. \emph{iii)} If $A \in \cO_n$ and $B \in \cO_m$, then $A \times B \in \cO_{n+m}$. \emph{iv)}  If $A \in \cO_{n+1}$, then the projection of $A$ onto its first $n$ coordinates is in $\cO_n$. \emph{v)} Every element of $\cO_1$ is exactly a finite union of intervals and points of $\bbR$. 
\end{definition}

Sets contained in $\cO$ are called \emph{definable}. In the following, we fix an o-minimal structure $\cO$ and definable will always mean definable in $\cO$. For $B \subset \bbR^d$, a function $f: B \rightarrow \bbR^m$ is said to be definable if $\Graph f = \{(x,f(x)): x \in B \} \subset \bbR^{d+m}$ is definable, we note that this implies that $B$ is definable. Similarly, a set-valued map $D: B \rightrightarrows \bbR^m$ is said to be definable if $\Graph D = \{(x,v) : x \in B, v \in D(x)\} \subset \bbR^{d+m}$ is definable.

In this note, we will also look to extended-valued functions $f : \bbR^d \rightarrow \bbR \cup \{ + \infty\}$. The latter is definable if $\Graph f = \{ (x, f(x)): x \in \bbR^d \textrm{ and } f(x) \neq + \infty\}$ is definable. Note that $\dom f = \{x \in \bbR^d: f(x) \neq + \infty \}$ is the projection of $\Graph f$ onto the first $d$ coordinates and thus is definable by the fourth axiom of Definition~\ref{def:o-min}.

Definable sets and maps have remarkable stability
properties. For instance, if $f$ and $A$ are definable, then $f(A)$
and $f^{-1}(A)$ are definable and definability is stable by most of the common operators such as  $\{+, -, \times, \circ, \circ^{-1}\}$.

Let us look
at some examples of o-minimal structures.\\
\textbf{Semialgebraic.} Semialgebraic sets form an o-minimal structure. This follows from the celebrated result of Tarski \cite{tarski1951decision}. A set $A \subset \bbR^n$ is semialgebraic if it is a finite union of intersections of sets of the form $\{ Q(x) \leq 0\}$, where $Q : \bbR^n \rightarrow \bbR$ is some polynomial. A function is semialgebraic if its graph is a semialgebraic set. Examples of such functions include any piecewise polynomial and rational functions but also functions such as $x \mapsto x^{q}$, where $q$ is any rational number. In fact, \emph{any o-minimal structure} contains every semialgebraic set.\\
\textbf{Globally subanalytic.} There is an o-minimal structure that contains, for every $n \in \bbN$, sets of the form $\{ (x,t) : t = f(x)\}$, where $f : [-1, 1]^n \rightarrow \bbR$ is an analytic function. This comes from the fact that subanalytic sets are stable by projection, which was established by Gabrielov \cite{gabrielov1968projections, gabrielov1996complements}. The sets belonging to this structure are called globally subanalytic (see \cite{bier_semi_sub} for more details).\\
\textbf{Log-exp.} There is an o-minimal structure that contains, semialgebraic sets, globally sub-analytic sets as well as the graph of the exponential
and the logarithm (see \cite{wilkie1996model, van1994elementary}).

With these examples in mind it is usually easy to verify that a function is definable. This will be the case as soon as the function is constructed by a finite number of definable operations on definable functions. From this, we see that \emph{most of the functions} used in optimization are definable in the structure \emph{Log-exp}.

  A particularly attractive property of definable sets is that they always admit a stratification. In the following we will say that a stratification $\bbX = (\cX_i)$ is definable if every stratum $\cX$ is definable.
  
  \begin{proposition}[{\cite[Theorem 4.8]{van96}}]\label{prop:whitney_strat}
    Let $\{A_1, \dots, A_k\}$ be a family of definable sets of $\bbR^d$. For any $p \geq 1$, there is a definable Whitney $C^p$ stratification of $\bbR^d$ compatible with $\{A_1, \dots, A_k \}$.
   \end{proposition}
  Similarly, a definable function always admit a Whitney stratification.
   \begin{proposition}[{\cite[Theorem 4.8]{van96}}]\label{prop:whitney_func}
    Consider a closed definable set $A \subset \bbR^d$ and a definable function $f: A \rightarrow \bbR^m$. For any family $\{A_1, \ldots, A_k \}$ of definable subsets of $\bbR^d$ and any family $\{B_1, \ldots, B_{k'} \}$ of definable subsets of $\bbR^m$, there is $(\bbX, \bbF)$ a definable Whitney $C^p$ stratification of $f$ such that $\bbX$ is compatible with $\{A_1, \ldots, A_k\}$ and $\bbF$ is compatible with $\{B_1, \ldots, B_{k'}\}$. 
   \end{proposition}

   \begin{remark}
   The fact that $A$ is closed is actually superfluous. Indeed, let $\tilde{f} : \overline{A} \rightarrow \bbR^m$ be any definable extension of $f$ to $\overline{A}$. Then, we can stratify $\tilde{f}$ with the help of the previous proposition and $\bbX$ can be taken compatible with $\{A, \overline{A} \backslash A \}$.
   \end{remark}
   If the definable function is \emph{real-valued} and \emph{continuous}, then it admits a Thom stratification.
  
   \begin{proposition}[\cite{le1997thom}]\label{prop:f_thom}
    Consider a definable set $A \subset \bbR^d$ and a definable, continuous, function $f: A \rightarrow \bbR$. There is $\bbX$ a definable Whitney $C^p$ stratification of $\bbR^d$ compatible with $A$ such that $\bbX$ satisfies the Thom $(a_f)$ condition. Moreover, $\bbX$ can be taken compatible\footnote{The result in \cite{le1997thom} is stated without compatibility conditions. However, examining its proof we see that it establishes compatibility, see Remark~\ref{rmk:stand_strat} in Section~\ref{sec:pf_prop_proj_con}.} with any finite collection of definable sets.
    \end{proposition}
   For technical reasons, we record here another result on a stratification of the graph of a definable function.
   \begin{proposition}\label{prop:whitn_graph}
    Consider a definable set $A \subset \bbR^d$ and a definable function $f: A \rightarrow \bbR^m$. There is $\bbG = (\cG_i)$, a definable Whitney $C^p$ stratification of $\Graph f$, and $\bbX = (\cX_i)$, a definable Whitney $C^p$ stratification of $A$, such that for any $ \cX \in \bbX$, $f_{|\cX}$ is $C^p$ and there is $\cG \in \bbG$ such that $\Graph f_{|\cX} =\cG$. Moreover, $\bbX$ and $\bbG$ can be taken compatible with any finite collection of definable sets.
   \end{proposition}
   \begin{proof}
    In this proof, we denote $\Pi^x : \bbR^{d+m} \rightarrow \bbR^d$, the restriction onto the first $d$ coordinates.
  
    Consider $\bbX' = (\cX')$, a Whitney $C^p$ stratification of $A$, such that $f$ is $C^p$ on each $\cX' \in \bbX'$.
  Applying Proposition~\ref{prop:whitney_func} to $\Pi^x$, there is $\bbG = (\cG_i)$ a Whitney $C^p$ stratification of $\bbR^{d+m}$ compatible with $\Graph f$ and $\bbX = (\cX_i)$ a Whitney $C^p$ stratification of $\bbR^d$, compatible with $\bbX'$, such that for any $\cG \in \bbG$, there is $\cX \in \bbX$ for which $\Pi^x(\cG) = \cX$. Moreover, $\bbX$ and $\bbG$ can be taken compatible with any finite collection of definable sets. 
  Then, $\bbX,\bbG$ satisfy the claim.

   Indeed, given any $\cX \in \bbX$, with $\cX \subset A$, there is \emph{i)} $\cX' \in \bbX'$ such that $\cX \subset \cX'$, therefore $f_{|\cX}$ is $C^p$, and \emph{ii)} there is $\cG \in \bbG$ such that $\cG \subset \Graph f$ and $\Pi_x(\cG) = \cX$, which implies that $\Graph f_{|\cX} = \cG$.
   \end{proof}

\subsection{Conservative set-valued fields}\label{sec:prel_consfield}

Conservative set-valued fields were introduced by Bolte and Pauwels in \cite{bolte2021conservative} as an elegant description of automatic differentiation of the loss function of a neural network through the application of the backpropagation algorithm. Since then, several works have worked out some geometrical properties of conservative set-valued mappings of definable functions (\cite{lewis2021structure,davis2022conservative,pauwels2023conservative}). They constitute an important tool for establishing the convergence of first-order methods in nonsmooth optimization (\cite{dav-dru-kak-lee-19, bolte2021conservative,bolte2023one,bolte2024differentiating,xiao2023adam, le2023nonsmooth}).

\begin{definition}[\cite{bolte2021conservative}]\label{def:cons_f}
  We say that a locally bounded, closed set-valued map $D: \bbR^d \rightrightarrows \bbR^d$ with nonempty values is a \emph{conservative field} for a \emph{potential function} $f: \bbR^d \rightarrow \bbR$ if for any absolutely continuous curve $\sx: [0, 1] \rightarrow \bbR^d$ and any measurable function $\sv: [0,1] \rightarrow \bbR^d$, such that for all $t \in [0,1]$, $\sv(t) \in D(\sx(t))$, it holds that
  \begin{equation}\label{eq:cons_field}
   f(\sx(1)) = f(\sx(0)) +  \int_{0}^1 \scalarp{\sv(t)}{\dot{\sx}(t)} \dif t \, .
  \end{equation}
  Functions that are potentials of some conservative field are called \emph{path differentiable}.
\end{definition}

Definable functions always admit a conservative field. The most important example of one is the Clarke subgradient (this was proven in \cite{drusvyatskiy2015curves} but see also \cite{dav-dru-kak-lee-19}). Recall that for a set $A$, $\conv A$ denotes its convex hull.

\begin{definition}[the Clarke subgradient {\cite{cla-led-ste-wol-livre98}}]
  Let $f: \bbR^d \rightarrow \bbR$ be a locally Lipschitz function. The Clarke subgradient of $f$ at $x$ is defined as
  \begin{equation*}
    \partial f(x) := \conv \{ v \in \bbR^d: \textrm{ there is $x_n \rightarrow x$, with $f$ differentiable at $x_n$ and $\nabla f(x_n) \rightarrow v$}\}\, .
  \end{equation*}
\end{definition}

From the optimization perspective, note that the Clarke subgradient provides a necessary condition of optimality: if $x$ is a local minimum of $f$, then $0 \in \partial f(x)$ (\cite{cla-led-ste-wol-livre98}).
\begin{proposition}[{\cite[Corollary 1]{bolte2021conservative}}]
  Let $f:\bbR^d \rightarrow \bbR$ be a definable locally Lipschitz continuous function. Then $\partial f : \bbR^d \rightrightarrows \bbR^d$ is a conservative set-valued field for $f$. Moreover, if $D$ is any other conservative field of $f$, then so is $ x \rightrightarrows \conv D(x)$ and 
  \begin{equation*}
    \partial f(x) \subset \conv D(x)\, .
  \end{equation*}
\end{proposition}
In particular, a point $x \in \bbR^d$ is a local minimum only if $0 \in \conv D(x)$.

  The notion of conservativity readily extends to the case where the potential function is vector-valued.
  \begin{definition}[\cite{bolte2021conservative}]\label{def:cons_map}
    Let $f: \bbR^d \rightarrow \bbR^m$ be a locally Lipschitz continuous function.  We say that a locally bounded, closed set-valued map $D: \bbR^d \rightrightarrows \bbR^{m \times d}$, with nonempty values is a \emph{conservative mapping} for $f$, if for any absolutely continuous curve $\sx : [0,1] \rightarrow \bbR^d$, and any measurable function $\sJ : [0,1] \rightarrow \bbR^{m \times d}$, such that for all $t \in [0,1]$, $\sJ(t) \in D(\sx(t))$, it holds that
    \begin{equation*}
      f(\sx(1)) = f(\sx(0)) + \int_{0}^1 \sJ(\sx(t))\dot{\sx}(t) \dif t \, .
    \end{equation*}
  \end{definition}
  
  Of course, when $m=1$, this definition is exactly the one of a conservative set-valued field. Not surprisingly, rows of a conservative mapping are actually conservative fields for the corresponding coordinate of $f$ (see \cite[Section 3.3]{bolte2021conservative}).

  As we will see in the next section, any definable conservative mapping admits a transparent geometric structure: a \emph{variational stratification}.
\section{Variational stratifications}\label{sec:var_strat}
From a geometric perspective, it turns out that in the definable setting conservative mappings are exactly those maps that admit a \emph{variational stratification}.
\begin{definition}[{\cite[Definition 5]{bolte2021conservative}}]\label{def:var_strat}
  Consider $p \geq 1$, $B \subset \bbR^d$, $f: B \rightarrow \bbR^m$ and $D: \bbR^{d} \rightrightarrows \bbR^{m \times d}$. We say that the pair $(f,D)$ admits a \emph{$C^p$ variational stratification} if there is a $C^p$ Whitney stratification $\bbX = (\cX_i)$ of $\bbR^d$, compatible with $B$, such that $f$ is $C^p$ on each stratum and for every $x \in \cX \subset B$,
  \begin{equation}\label{eq:def_max_cons}
    \begin{split}
      D(x) \subset \Jac_{\cX} f(x) + \cR^m_{\cX}(x)\, ,
    \quad  \textrm{ with } \cR^m_{\cX}(x) = \{R \in \bbR^{m \times d} : \cT_x \cX \subset \ker R\}\, .
    \end{split}
  \end{equation}
  If $f, D$ and $\bbX$ are definable, then we will say that $(f,D)$ admits a \emph{definable} $C^p$ stratification.
\end{definition}
\begin{remark}
  In other words, for every $x \in \cX$ and $J \in D(x)$,
  \begin{equation*}
    J = \Jac_{\cX} f(x) + [u_1, \ldots, u_m]^{\top}\, ,
  \end{equation*}
  where for $1\leq i \leq m$, $u_i \in \cN_{\cX}(x)$. Equivalently, for every $h \in \cT_{\cX}(x)$, $Jh = \Jac_{\cX}f(x)h$.
\end{remark}

\begin{remark}
Note that for a fixed stratification $C^p$ Whitney stratification $\bbX = (\cX_i)$, the mapping $x \rightrightarrows \cR^m_{\cX}(x)$, is a conservative mapping of the zero function. This mapping previously appeared in \cite{bolte2021nonsmooth} under the name of residual.
\end{remark}
\begin{remark}
  By a slight abuse of notations, when $f$ is real-valued, we will consider mappings $D : \bbR^d \rightrightarrows \bbR^{d}$ (as opposed to $\bbR^{d} \rightrightarrows \bbR^{1 \times d}$.) In this case, Equation~\eqref{eq:def_max_cons} becomes 
  \begin{equation*}
  D(x) \subset \nabla_{\cX} f(x) + \cN_{\cX}(x) \, .
  \end{equation*}
  This is analogous to the fact that for differentiable real-valued functions it is sometimes more convenient work with the gradient $\nabla f(x) \in \bbR^d$ than with the Jacobian $\Jac f(x)  = (\nabla f(x))^{\top}\in \bbR^{1 \times d}$.
\end{remark}
\begin{remark}
  In optimization we sometimes consider extended-valued functions $f: \bbR^d \rightarrow \bbR \cup \{-\infty, + \infty \}$. In such case, denoting $ \dom f = \{x: x \in \bbR^d\, , |f(x)| \neq + \infty \}$, we will say that $(f,D)$ admits a $C^p$ variational stratification if $(f_{|\dom f}, D)$ admits a $C^p$ variational stratification.
\end{remark}

Note that, differently to \cite{bolte2021conservative}, we do not impose on $D$ to be defined everywhere or have compact values. In particular, it is allowed to be empty-valued.

The link between conservative mappings and variational stratifications is given by the following result.
\begin{proposition}[{\cite[Theorem 2.2]{lewis2021structure}} and \cite{pauwels2023conservative,davis2022conservative}]\label{prop:cons_strat}  
    Let $f: \bbR^d \rightarrow \bbR^m$ be definable, locally Lipschitz continuous, and let $D: \bbR^d \rightrightarrows \bbR^{m \times d}$ be a definable conservative mapping of $f$. The pair $(f,D)$ admits a $C^p$ variational stratification for any $p \geq 1$.
    
    Conversely, if there exists a closed, nonempty-valued, locally bounded mapping $D: \bbR^d \rightrightarrows \bbR^{m \times d}$ such that $(f,D)$ admits a $C^p$ variational stratification, then $D$ is a conservative set-valued field of $f$.
  \end{proposition}

Therefore, Definition~\ref{def:var_strat} is strictly more general than Definitions~\ref{def:cons_f} and \ref{def:cons_map}. As we show in the next paragraph, the notion of a variational stratification allows us to work with potentials $f$ that are not necessarily Lipschitz continuous\footnote{We note that in Definition~\ref{def:var_strat} the function $f$ is not even required to be continuous. It is however $C^p$  around any point in a stratum of full dimension, hence, almost everywhere. Furthermore, as shows the example of $f(x) = \1_{x \geq 0}$ and $D(x) \equiv \{0\}$, $f$ does not need to admit a weak derivative in the sense of Sobolev.} or with set-valued maps that are not locally bounded or even defined everywhere. Furthermore, in Section~\ref{sec:def_limit} we will establish that the gradient's limit will admit a variational stratification, which, depending on the regularity of $(f_a)$, can be empty-valued at some points.

\paragraph{Examples.} To present some interesting examples, we first need to define various notions of tangent and normal cones of a set.
For a set $B \subset \bbR^d$, the tangent cone is defined as
\begin{equation*}
  \cT^{F}_{B}(x) =\{ v \in \bbR^d : \textrm{ there is $(t_n, x_n) \rightarrow (0, x)$, with $t_n >0$, $x_n \in B$ and $(x_n - x)/t_n \rightarrow v$} \}\, .
\end{equation*}
The Frechet normal cone, $\cN_B^{F} : \bbR^d \rightrightarrows \bbR^d$, is defined as $\cN_{|\bbR^d \backslash B}\equiv \emptyset$ and for $x \in B$,
\begin{equation*}
  \cN^{F}_B(x) = \{ v \in \bbR^d: v^{\top}u \leq 0 \textrm{ for all $u \in \cT_{B}(x)$} \} \footnote{If $B$ is a $C^p$ manifold, then $\cT_{B}(x)$, $\cN_B(x)$ are simply the tangent and the normal planes of $B$ at $x$.}\, .
\end{equation*}
The limiting normal cone is then defined as 
\begin{equation*}
  \cN^{L}_B(x) = \{ v \in \bbR^d: \textrm{there is $(x_n, v_n) \rightarrow (x,v)$, $x_n \in B$ and $v_n \in \cN^{F}_B(x_n)$}\}\, .
\end{equation*}
Finally, the Clarke normal cone is 
\begin{equation*}
  \cN^{c}_B(x) = \overline{\conv} \cN^{L}_{B}(x) \, ,
\end{equation*}
where $\overline{\conv}$ denotes the closure of the convex hull.
The following proposition will be a simple consequence of results of Section~\ref{sec:clos_graph_consmpa} (see Remark~\ref{rmk:proof_ofprop} for a proof).
\begin{proposition}\label{prop:defvar_normalcone}
  Assume that the set $B$ is definable. Let $\iota_{B} : \bbR^d \rightarrow \bbR \cup \{+ \infty \}$ be the characteristic function: $\iota_{B} \equiv 0$ on $B$ and $\iota_{B} \equiv + \infty$ on $\bbR^d \backslash B$. Then, for any $p \geq 1$, $(\iota_{B}, \cN^{F}_B)$, $(\iota_{B}, \cN^{L}_B)$ and $(\iota_{B}, \cN^{c}_B)$ admit a definable $C^p$ variational stratification.
\end{proposition}

Now, consider a continuous definable function $f: \bbR^d \rightarrow \bbR$. We define the Frechet, limiting, horizontal and Clarke subgradients of $f$ at $x\in \bbR^d$ as follows. 
\begin{equation*}
  \begin{split}
  \partial_F f(x) &= \{ v\in\bbR^d: \forall y \in \bbR^d\, , f(y) - f(x) \geq \scalarp{v}{y-x} + o(\norm{y-x})\}\,, \\
  \partial_L f(x) &= \{ v\in\bbR^d: \textrm{ there is $(x_n, v_n) \rightarrow (x,v)$, $v_n \in \partial_F f(x_n)$}\}\,, \\
  \partial^{\infty} f(x) &= \{ v \in \bbR^d :\textrm{ there is $(x_n, t_n, t_n v_n) \rightarrow (x,0, v)$, $v_n \in \partial_F f(x_n)$, $t_n >0$}\}\,, \\
  \partial f(x) &= \overline{\conv} (\partial_L f(x) + \partial^{\infty} f(x))\, .
\end{split}
\end{equation*}
The following proposition is the celebrated projection formula proved in \cite{bolte2007clarke}.
\begin{proposition}
  If $f: \bbR^d \rightarrow \bbR$ is continuous and definable, then the pairs $(f, \partial_F f)$, $(f, \partial_L f)$, $(f, \partial f)$ admit a definable $C^p$ variational stratification.
\end{proposition}
\begin{remark}
  Considering $f(x)=  \sign(x) \sqrt{|x|}$, we see that all of these subgradients may be empty-valued: $\partial f(0) = \emptyset$. Moreover, they are non locally bounded. Thus, formally speaking, they are not conservative fields of $f$.
\end{remark}

The rest of this section is divided in two parts. In the first, we show that in the definable setting, Definition~\ref{def:var_strat} is versatile enough to preserve simple computational properties such as stability by union, convexity of composition. 
In the second, we establish an important property of definable variational stratifications: stability by graph's closure. 

To understand the motivation behind such result let us return to the problem stated in the introduction.
If the function $f:(x,a) \mapsto f_a(x)$ turned out to be differentiable, then obviously $\nabla f_a(x) = \Pi^x(\nabla f(x,a))$, where $\Pi^x : \bbR^{d+1} \rightarrow \bbR^d$ is the restriction onto the first $d$ coordinates. Hence, if $f$ is $C^1$, then automatically $D(x) = \nabla F(x)$. This observation indicates a potential approach to solve the problem: construct a mapping $\tilde{D} : \bbR^{d+1} \rightrightarrows \bbR^{d+1}$, such that for any $a$, $\partial f_a(x) = \Pi^x(\tilde{D}(x,a))$. If such mapping $\tilde{D}$ was constructed, $D_F$ can be interpreted as some kind of ``limit" of $\partial f_a = \Pi^x(\tilde{D}(x,a))$, when $a \rightarrow 0$ (see Theorem~\ref{thm:graph_closed_param} for the exact definition of the limit). Thus, we need to investigate the stability of the variational stratification property under such limit, which is exactly the content of Theorem~\ref{thm:graph_closed_param}.

\subsection{Computational properties of definable variational stratifications}\label{sec:comp_prop_defvar}

In this section, we establish some simple properties of variational stratifications.
In the following, we fix a set $B \subset \bbR^d$, a function $f: B \rightarrow \bbR^m$ and  $p\geq1$.

First, if $f$ is definable, then there is always a non-trivial mapping $D$, such that $(f,D)$ admits a variational stratification.

\begin{lemma}
  If $f$ is definable, then there is a definable map $D: \bbR^d \rightrightarrows \bbR^{m\times d}$ such that for all $x \in B$, $D(x) \neq \emptyset$ and $(f,D)$ admits a definable $C^p$ variational stratification.
\end{lemma}
\begin{proof}
  There is  $(\cX_i)$, a $C^p$ Whitney stratification of $\bbR^d$, compatible with $B$, such that $f$ is $C^p$ on each stratum. The map $D(x) = \Jac_{\cX} f(x)$ if $x \in \cX \subset B$ and $D(x) = \emptyset$ otherwise, satisfies the claim.
\end{proof}
The next lemmas establish that the concept of variational stratification is stable by union.
\begin{lemma}\label{lm:strat_adapt}
  Let the pair $(f,D)$ admit a definable $C^p$ variational stratification $\bbX$. If $\bbX'$ is a Whitney $C^p$ stratification, compatible with $\bbX$, then $\bbX'$ satisfies the requirements of Definition~\ref{def:var_strat}.
\end{lemma}
\begin{proof}
  Consider $\cX' \in \bbX'$. By assumption there is $\cX \in \bbX$ such that $\cX' \subset \cX$. Thus, for any $x \in \cX'$, $\cT_x \cX' \subset \cT_x \cX$. Therefore, $f$ is $C^p$ on $\cX'$ and for all $h \in\cT_x \cX' $ and $J \in D(x)$,  $\Jac_{\cX'} f(x) h =  \Jac_{\cX} f(x)h = Jh$.
\end{proof}
\begin{lemma}\label{lm:common_whitney}
  Let $D_1, D_2 :  \bbR^{d} \rightrightarrows \bbR^{m \times d}$ be two definable maps such that $(f,D_1)$ and $(f, D_2)$ admit a definable $C^p$ variational stratification. Then, $(f, D_1 \cup D_2)$ admits a definable $C^p$ variational stratification.
\end{lemma}

\begin{proof}
  Let $\bbX_1$ (respectively $\bbX_2$) be the $C^p$ Whitney stratification of $\bbR^d$ associated with $D_1$ (respectively $D_2$). By definability, there is $\bbX_3$ a Whitney $C^p$ stratification that is compatible with $\bbX_1$ and $\bbX_2$. This implies the statement by Lemma~\ref{lm:strat_adapt}.
\end{proof}

From Definition~\ref{def:var_strat} we immediately obtain that variational stratifications are stable by convex closure. We recall that for a set $A$, $\conv(A)$ denote its convex hull.
\begin{lemma}
  Assume that the pair $(f,D)$ admits a $C^p$ variational stratification. Define
  \begin{equation*}
    \conv D : x \rightrightarrows \conv (D(x))\, ,
  \end{equation*}
  Then, the pair $(f, \conv(D))$ admits a $C^p$ variational stratification. 
\end{lemma}

Rows of conservative mappings are conservative fields. We have an analogous result for variational stratifications.
\begin{lemma}\label{lm:cons_map_fiel_var}
  Denote $f = (f_1, \ldots, f_m)$ and let $D: \bbR^{d} \rightrightarrows \bbR^{m \times d}$ be such that $(f,D)$ admits a $C^p$ variational stratification. For $1\leq i \leq m$, denote $D_i$ the $i$-th row of $D$.
  Then, $(f_i, D_i)$ admits a $C^p$ variational stratification.

  Conversely, if for every $i \leq m$, $(f_i, D_i)$ admits a definable $C^p$ variational stratification, then $(f,[D_1, \ldots, D_m]^{\top})$ admits a definable $C^p $variational stratification.
\end{lemma}
\begin{proof}
  If $\cX \subset \bbR^d$ is such that $f_{|\cX}$ is $C^p$, then $\Jac_{\cX} f(x) = [\nabla_{\cX} f_1(x), \ldots, \nabla_{\cX} f_m(x)]^{\top}$ and the first point is obvious.
  
  To prove the second point denote $\bbX_i$, the definable $C^p$ stratification of $(f_i, D_i)$. By definability, there is $\bbX$ a Whitney $C^p$ stratification compatible with $\bbX_1, \ldots, \bbX_m$. For any $x \in \cX \in \bbX$ and $h \in \cT_{\cX}(x)$, $D(x)h = [D_1(x), \ldots, D_m(x) ]^{\top} h = [\nabla_{\cX} f_1(x), \ldots, \nabla_{\cX} f_m(x)]^{\top} h = \Jac_{\cX} f(x) h$.
\end{proof}

Finally, the notion of definable variational stratification is stable through compositions.
\begin{proposition}\label{pr:comp_var}
  Consider a definable set $B_g \subset \bbR^m$ and a function $g: B_g \rightarrow \bbR^l$ such that $f(B) \subset B_g$. Consider $D_f: \bbR^d \rightrightarrows \bbR^{m \times d}$ and $D_g: \bbR^{m} \rightrightarrows \bbR^{l \times m}$ two definable set-valued maps such that $(f, D_f)$ and $(g,D_g)$ admit definable $C^p$ variational stratifications. Defining  $D_g D_f : \bbR^{d} \rightrightarrows \bbR^{l \times d}$ as
  \begin{equation*}
   D_g D_f :x \rightrightarrows \{ J_g J_f : J_f \in D_f(x) \, , J_g \in D_g(f(x)) \}\, ,
  \end{equation*} 
the pair $(g\circ f, D_g D_f)$ admits a $C^p$ variational stratification.
\end{proposition}

\begin{proof}
  Let $\bbX_f = (\cX_i)$ be a stratification for $(f,D_f)$ and $\bbM_g = (\cM_i)$ be a stratification for $(g, D_g)$. We can find $\bbX'_f$ a stratification of $\bbR^d$ compatible with $\bbX_f$ and $\{ f^{-1}(\cM) : \cM \in \bbM_g\}$. Then, for any $\cX \in \bbX'_f$, there is $\cM \in \bbM_g$ such that $f(\cX) \subset \cM$. Therefore, $g \circ f$ is $C^p$ on $\cX$ and, since for any $(x,h) \in (\cX,\cT_{x} \cX)$, $\Jac_{\cX} f h \in \cT_{f(x)} \cM$, we obtain for any $J_f \in D(x)$ and $J_g \in D(f(x))$,
  \begin{equation*}
   \Jac_{\cX}  g\circ f(x) [h] =  \Jac_{\cM} g(f(x))\left[ \Jac_{\cX} f(x) [h]\right] = J_g J_f h\, .
  \end{equation*}
\end{proof}

\subsection{Closing the graph of a conservative mapping}\label{sec:clos_graph_consmpa}

\subsubsection{Closure preserves conservativity}
We first note that closing the graph of a conservative mapping, preserve conservativity. As we discuss in the Remark~\ref{rmk:proj_bolte} below this result is closely related to the projection formula of \cite{bolte2007clarke}.

In this section, we fix a set $B\subset \bbR^d$, a function $f: B \rightarrow \bbR^m$, $p\geq1$ and a mapping $D : \bbR^{d} \rightrightarrows \bbR^{m \times d}$. We recall that $\iota_{B} : \bbR^d \rightarrow \bbR \cup \{+ \infty \}$ is defined as $\iota_{B}(x) = 0$, if $x \in B$ and $\iota_B(x) = + \infty$ otherwise.

\begin{proposition}\label{prop:graph_closed}
  Assume that the pair $(f,D)$ admits a definable $C^p$ variational stratification. Define
  \begin{equation*}
    \bar{D}: x \rightrightarrows  \{J \in \bbR^{m \times d}: \textrm{ there is }(x_n,f(x_n), J_n) \rightarrow (x,f(x), J) \textrm{ with $J_n \in D(x_n)$} \}\, ,
  \end{equation*}
  and
  \begin{equation*}
    \begin{split}
    D_{\infty}: x\rightrightarrows \{ R \in\bbR^{m \times d} : &\textrm{ there is }(x_n,f(x_n), t_n, t_n J_n) \rightarrow (x,f(x), 0, R)\, ,\\
     &\textrm{ with $t_n \in \bbR$ and $J_n \in D(x_n)$}\}\, .
  \end{split}
  \end{equation*}
  Then, the pairs $(\iota_{B}, D_{\infty})$, $(f, \bar{D})$, $(f, \bar{D} + D_{\infty})$ and $(f, \conv(\bar{D} + D_{\infty}))$\footnote{For two sets $A,B$, $A+B =\{ a+b : a \in A, b \in B\}$. Note that if $\bar{D}(x) = \emptyset$, then $\bar{D}(x) + D_{\infty}(x) = \emptyset$.} admit a definable $C^p$ variational stratification.
\end{proposition}
\begin{proof}
  This can be proven in a manner similar to \cite[Proof of Proposition 4]{bolte2007clarke} but is also a consequence of the more general Theorem~\ref{thm:graph_closed_param} below.
\end{proof}
\begin{remark}
  Drawing an analogy with the case where $D$ is a subgradient of some real-valued function $f$, $\bar{D}$ can be interpreted as the limiting conservative field and $D_{\infty}$ as the horizontal one (see also Remark~\ref{rmk:proj_bolte}). We also note that if $f$ is continuous at $x$, then $\bar{D}$ admits a simpler expression
  \begin{equation*}
    \bar{D}(x) = \{ J \in \bbR^{m \times d} : \textrm{ there is }(x_n, J_n) \rightarrow (x, J) \textrm{ with $J_n \in D(x_n)$} \}\, .
  \end{equation*}
  That is $\bar{D}$ is simply the graph's closure of $D$ around $x$.
\end{remark}

\begin{remark}\label{rmk:proj_bolte}
  The celebrated projection formula of \cite{bolte2007clarke} shows that for a continuous, definable function $f: \bbR^d \rightarrow \bbR$, the Frechet (respectively the limiting, the Clarke) subgradient $\partial_F f$ (respectively $\partial_L f$, $\partial f$) admits a definable variational stratification. In the context of the previous proposition, if we denote the horizontal subgradient by $\partial_{\infty} f$, then choosing $D = \partial_F f$, we obtain $\bar{D} = \partial_L f$, $\partial_{\infty} f  \subset D_{\infty}$, and $\partial f \subset \conv(\bar{D} + D_{\infty})$.
\end{remark}
\begin{remark}\label{rmk:proof_ofprop}
Proposition~\ref{prop:defvar_normalcone} is a simple consequence of Proposition~\ref{prop:graph_closed}. Indeed,
  consider $\bbX = (\cX_i)$ a stratification of $\bbR^d$ compatible with $B$. For $x \in \cX \in \bbX$ such that $\cX \subset B$, it holds that $\cT_{\cX}(x) \subset \cT^F_{B}(x)$. Hence, $\cN^{F}_B(x) \subset \cN_{\cX}(x)$ and $\bbX$ is a $C^p$ variational stratification for the pair $(\iota_{B}, \cN^{F}_B)$. The result on $\cN^{L}_{B}$ and $\cN^{c}_{B}$ follows from Proposition~\ref{prop:graph_closed}.
\end{remark}

\subsubsection{Parametric closure}
Finally, we establish a parametrized version of Proposition~\ref{prop:graph_closed}. As discussed above, this is the main ingredient in our proof of convergence of gradients.

Setting up the stage, we consider a set $B \subset \bbR^{d+1}$, a function $f: B \rightarrow \bbR^m$ and a map $D_f: \bbR^{d+1} \rightrightarrows \bbR^{m \times (d+1)}$. For $a \in \bbR$, define 
\begin{equation*}
  B_a = \{ x: (x,a) \in B\}\, , \quad f_a(\cdot) = f(\cdot, a) \, 
\end{equation*}
and
\begin{equation*}
  D_a:= \{ J \in \bbR^{m\times d} : \textrm{there is $\tilde{J} \in D_f(x,a)$ and $v \in \bbR^m$ such that $\tilde{J} = [J, v]$}\}\, .
\end{equation*}
Finally, denote
\begin{equation*}
  F : B_0 \rightarrow \bbR^m \, ,\quad \textrm{ with } \quad F(\cdot) = f(\cdot, 0)\, .
\end{equation*} 

Note that $f_a$ can be viewed as a composition $x \mapsto (x,a) \mapsto f(x,a)$. Therefore, if $(f,D_f)$ admits a definable $C^p$ variation stratification, then by Proposition~\ref{pr:comp_var}, the same holds for the pair $(f_a,D_a)$. Thus, naturally one might wonder if the ``limit mapping" of $(D_a)_{a >0}$, when $a \rightarrow 0$, admits a $C^p$ variational stratification associated with $F$. Theorem~\ref{thm:graph_closed_param} positively answers this question. As we show below in Remark~\ref{rm:ces_multd_field}, such result does not necessarily hold when $a$ is not a one dimensional parameter. 

\begin{theorem}\label{thm:graph_closed_param}
  Assume that the pair $(f,D_f)$ admits a definable $C^p$ variational stratification.
   Define $\bar{D}_F, D_{F,\infty} : \bbR^d \rightrightarrows \bbR^{m \times d}$ as 
  \begin{equation*}
    \bar{D}_F: x \rightrightarrows  \{J \in \bbR^{m \times d}: \textrm{ there is }(x_n, a_n, f_{a_n}(x_n), J_n) \rightarrow (x,0, F(x), J) \textrm{ with $J_n \in D_{a_n}(x_n)$ } \}\, ,
  \end{equation*}
  and
  \begin{equation*}
    \begin{split}
    D_{F,\infty}: x\rightrightarrows \{ R \in\bbR^{m \times d} : &\textrm{ there is }(x_n, a_n,f_{a_n}(x_n), t_n, t_n J_n) \rightarrow (x,0, F(x), 0, R)\, ,\\
     &\textrm{ with $t_n \in \bbR$ and $J_n \in D_{a_n}(x_n)$}\}\, .
  \end{split}
  \end{equation*}
  Then, the pairs $(\iota_{B_0},D_{F,\infty})$ $(F, \bar{D}_F)$, $(F, \bar{D}_F + D_{F,\infty})$ and $(F, \conv(\bar{D}_F + D_{F,\infty}))$ admit a definable $C^p$ variational stratification.
\end{theorem}
\begin{remark}
  We can view $F(x)$ as a composition of $x \mapsto (x,0) \mapsto f(x,0)$. Thus, using Proposition~\ref{pr:comp_var}, it is tempting to say that Theorem~\ref{thm:graph_closed_param} is a consequence of Proposition~\ref{prop:graph_closed}. Unfortunately, an element of $\bar{D}_F(x)$ is not necessarily a restriction of an element of $\bar{D}(x,0)$ to the first $d$ columns. For example, if $f(x,a) = 2\sqrt{a}$ and for $a \neq 0$, $D(x,a) = \{(0, a^{-1/2})\}$, then $\bar{D}(x,0) = \emptyset$, while $\bar{D}_F(x)=  \{ 0\}$.
\end{remark}

\begin{proof}
    In this proof, we denote by $\Pi^{a}: \bbR^{d+1 + m} \rightarrow \bbR$, the function $\Pi^{a}(x,a,f) = a$ and by $\Pi^{x}: \bbR^{d+1+m} \rightarrow \bbR^d$ the function $\Pi^{x}(x,a,f) = x$. Note that for any $C^p$ submanifold $\cG \subset \bbR^{d+1+m}$, $\Pi^{a}_{|\cG}$ is $C^p$, with $\dif \Pi^{a}_{|\cG}(x,a,f)[h_x, h_a, h_f] = h_a$. Therefore, 
    \begin{equation*}
      \ker \dif \Pi^a_{|\cG}(x,a,f)  = \cT_{\cG}(x,a,f) \cap \bbR^{d} \times \{ 0\} \times \bbR^m \, .
    \end{equation*}
  
  The result will be a consequence of the following claim.
  \emph{Claim: There are $\bbM, \bbX, \bbG$ definable $C^p$ stratifications of respectively $\bbR^d$, $\bbR^{d+1}$, $\bbR^{(d+1)+ m}$ such that the following holds. 1) $\bbX$ is a definable $C^p$ variational stratification of the pair $(f,D)$. 2) $\bbG$ is compatible with $\{ \Graph f_{|\cX} : \cX \in \bbX\}$ and is Thom $(a_f)$ for $\Pi^a$. 3) $\bbM$ is compatible with $B_0$, and for any $\cM \in \bbM$, such that $\cM \subset B_0$, there is $\cX \in \bbX$ and $\cG \in \bbG$, for which $\Graph f_{|\cM \times \{ 0\}} \subset \cG \subset \Graph f_{|\cX}$.}
  
  Let us first see how this establishes the theorem. Consider $x \in \cM \in\bbM$, $h \in \cT_{\cM}(x)$ and a sequence 
  $(x_n, a_n, f_{a_n}(x_n), t_n, t_nJ_n) \rightarrow (x, 0, F(x), t, R )$, with $J_n \in D_{a_n}(x_n)$. Our goal is to prove that $R h = t\Jac_{\cM} F(x) h$. Indeed, if this was proven, then choosing $t_n \equiv 1$ and since $h \in \cT_{\cM}(x)$ was arbitrary, we obtain $\bar{D}_F(x) \subset \Jac_{\cM}F(x) + \cR^{m}_{\cX}(x)$. Similarly, choosing $t_n \rightarrow 0$, we obtain $D_{F, \infty}(x) \subset \cR^{m}_{\cX}(x)$. Therefore, $\bbM$ is a $C^p$ variational stratification for $(F,D_{F,\infty})$ $(F, \bar{D}_F)$, $(F, \bar{D}_F + D_{F,\infty})$ and $(F, \conv(\bar{D}_F + D_{F,\infty})$.
  
  Since there is $\cX_0 \in \bbX$, $\cG_0 \in\bbG$ such that $\cM \times \{ 0\} \subset \cX_0$, and $\Graph f_{|\cM \times \{ 0\}} \subset \cG_0 \subset \Graph f_{|\cX_0}$, we obtain that $(h,0) \in \cT_{\cM\times \{0\}}(x,0) \subset \cT_{\cX_0}(x, 0)$ and 
  \begin{equation*}
    (h,0, \Jac_{\cM}F(x) h) = (h,0, \Jac_{\cX_0} f(x,0)[h, 0])  \subset \cT_{\cG_0}(x,0, F(x)) \cap  \bbR^{d} \times \{ 0\} \times \bbR^m\, .
  \end{equation*}
  This shows that
    $(h,0, \Jac_{\cM}F(x) h) \in \ker \dif \Pi^{a}_{|\cG_0}(x,0, F(x))$. 
  Without loss of generality, we can assume that $(x_n, a_n, f_{a_n}(x_n))$ lies in a unique stratum $\cG_1 \in \bbG$. Then, by Thom ($a_f$) property, there is a sequence $(h_x^n,h_a^n, h_f^n) \in \ker \dif \Pi^{a}_{|\cG_1}(x_n, a_n, f_{a_n}(x_n))$ such that 
    $(h_x^n, h_a^n, h_f^n) \rightarrow (h, 0, \Jac_{\cM}F(x) h)$.
  We claim that $h_f^n = J_n h_x^n$. Indeed, there is $\cX_1 \in \bbX$ such that $\cG_1 \subset \Graph f_{|\cX_1}$. Thus, $ \ker \dif \Pi^{a}_{|\cG_1}(x_n, a_n, f(x_n,a_n)) \subset \{  (h_x, 0, J_n h_x) : (h_x,0) \in \cT_{\cX_1}(x_n, a_n)\}$ 
  and $J_n h_x^n = h_f^n$ as claimed. Finally, we obtain that $
   Rh  = \lim_{n \rightarrow + \infty } t_n J_n h_x^n = \lim_{n \rightarrow + \infty} t_n h_f^n= t \Jac_{\cM} F(x) h$, 
  which establishes the theorem.
  
  \emph{Proof of the claim.}
    Consider $\bbX'$ a definable $C^p$ variational stratification of $\bbR^{d+1}$ associated to $(f,D)$. Consider $\bbX, \bbG'$ stratifications of respectively $\bbR^{d+1}$ and $\bbR^{d+1 + m}$, given by Proposition~\ref{prop:whitn_graph}, with $\bbX$ compatible with $\bbX'$. That is to say, for every $\cX \in \bbX$, there is $\cG' \in \bbG'$ such that $\Graph f_{|\cX} = \cG'$. By Proposition~\ref{prop:f_thom} there is $\bbG$ a Thom $(a_f)$ stratification for $\Pi^{a}$, compatible with $\bbG'$ and $\bbR^{d} \times \{ 0\} \times \bbR^{m}$. Finally, consider
    \begin{equation*}
      \tilde{\bbM} = \{ \Pi^{x}(\cG) \subset \bbR^d : \textrm{$\cG \in \bbG$, with $\cG \subset \Graph f \cap \bbR^{d} \times\{ 0\} \times \bbR^{m}$} \}\, 
    \end{equation*}
    and $\bbM$ a stratification compatible with $\tilde{\bbM}$.
  
    Now by construction $\bbX, \bbG$ satisfy the two first points of the claim. To show the third point, first note that any element of $\tilde{\bbM}$ is in $ B_0 = \Pi^{x}(\Graph f \cap \bbR^{d} \times\{ 0\} \times \bbR^{m} )$ and that $\tilde{\bbM}$ covers $B_0$.
  Therefore, \emph{i)} $\bbM$ is compatible with $B_0$, \emph{ii)} for any $\cM \in \bbM$, such that $\cM \subset B_0$, there is $\cG \in \bbG$, for which $\Graph f_{|\cM \times \{ 0\}} \subset \cG$. Since $\bbG$ is compatible with $\{ \Graph f_{|\cX} : \cX \in \bbX \}$, there is also $\cX \in \bbX$ such that $ \Graph f_{|\cM \times \{ 0\}} \subset \cG \subset \Graph f_{|\cX}$, which proves the claim.
\end{proof}

\begin{remark}\label{rm:ces_multd_field}
  It is important that in Theorem~\ref{thm:graph_closed_param} the parameter $a$ is one-dimensional. Indeed, consider $f : (-1,1) \times \bbR^2 \rightarrow \bbR$, with $f(x,a_1, a_2) \equiv 0$ and
  \begin{equation*}
    D(x,a_1, a_2)  =\begin{cases}\left \{ \left(1, \frac{a_2(a_1^2-a_2^2)}{2(a_1^2 + a_2^2)^2}, \frac{a_1(a_2^2-a_1^2)}{2(a_1^2 + a_2^2)^2}\right) \right \}
       \quad &\textrm{ if $(a_1, a_2) \neq (0,0)$ and $x = \frac{a_1 a_2}{2(a_1^2 + a_2^2)}$}\, , \\
       \{(0,0,0) \} \quad &\textrm{ otherwise}\, .
    \end{cases}
  \end{equation*} 
  Let $\cX_1 = (-1,1) \times \{ 0\} \times \{ 0\}$, $\cX_2 = \{(x, a_1, a_2) : 2x = a_1 a_2/(a_1^2 + a_2^2) \, , (a_1, a_2) \neq (0,0)\,, |x|<1 \}$ and $\cX_3 = \bbR^3 \backslash (\cX_1 \cup \cX_2)$. Then, $\cX_1,\cX_2$ is part of a definable $C^{p}$ (for any $p$) variational stratification of $(f,D)$. Through simple computations we obtain, for $x \in(0,1)$, $\bar{D}_F(x) = \{0, 1 \}$. Obviously, $\bar{D}_F(x)$ is not a conservative field of $F \equiv 0$. 
\end{remark}

\section{Definable limits of conservative fields}\label{sec:def_limit}
In the following, we fix an o-minimal structure $\cO$. Definable will always mean definable in $\cO$.

  Let us fix a set $B \subset \bbR^{d+1}$, a function $f: B \rightarrow \bbR^{m}$ and a set-valued map $D: \bbR^d \times \bbR\backslash \{0 \} \rightrightarrows \bbR^{m \times d}$. For every $a \in \bbR$, we denote 
  \begin{equation*}
  B_a = \{ x: (x,a) \in B\}\, , \quad f_a(\cdot) = f(\cdot, a) \, ,
  \end{equation*}
  and for $a \neq 0$, 
  \begin{equation*}
   D_a(\cdot) = D(\cdot, a)\, .
  \end{equation*}
  Furthermore, we define $F: B_0 \rightarrow \bbR^m$ as $F = f_0$. We think of $(f_a)_{a >0}$ as a parametrized family of functions and $(D_a)_{a >0}$ as family of corresponding conservative mapping. We will work under the following assumption.
 \begin{assumption}\label{hyp:new_proof}\phantom{=}
   \begin{enumerate}[i)]
     \item The mappings $f,D$ and the set $B \subset \bbR^{d+1}$ are definable.
     \item For every $a\neq 0$, $(f_a, D_a)$ admits a definable $C^p$ variational stratification.
    \end{enumerate}
 \end{assumption}
Let us emphasize that if for every $a$, $D_a \equiv \partial f_a$, with $\partial f_a$ denoting the Clarke subgradient, then the set-valued map $D$ is definable as soon as $f$ is. We record this fact in the next lemma, the proof of which is given in Appendix~\ref{sec:rem_pfs}.
\begin{lemma}\label{lm:subg_encompas}
  Assume that $B = \bbR^d \times [0, a_0]$, $f: B \rightarrow \bbR$ is definable and for every $a \in (0, a_0]$, $f_a$ is continuous and $D_a \equiv \partial f_a$. Then, $D$ is definable.
\end{lemma}
Therefore, Assumption~\ref{hyp:new_proof} encompasses the setting presented in the introduction. Nevertheless, we work in a more general setting, where, besides Assumption~\ref{hyp:new_proof} we do not impose any regularity, such as continuity, on $f$ and $D$.

Consider the set-valued maps $D_F, D_{F, \infty} : \bbR^d \rightrightarrows \bbR^{m \times d}$ defined as
  \begin{equation}\label{eq:def_DF}
    \begin{split}
    D_F(x) := \{J \in \bbR^{m \times d}: &\textrm{ there is } (x_n, a_n, f_{a_n}(x_n), J_n) \rightarrow (x,0, F(x), J)\, , \textrm{ with } J_n \in D_{a_n}(x_n)\}         
  \end{split}
  \end{equation}
  and 
  \begin{equation}\label{eq:def_DF_inf}
    \begin{split}
    D_{F, \infty} (x) = \{ R \in \bbR^{m \times d} : &\textrm{ there is } (x_n,a_n, f_{a_n}(x_n), t_n, t_nJ_n) \rightarrow (x,0, F(x), 0, R)\, ,\\
    &\textrm{ with } J_n \in D_{a_n}(x_n)\textrm{ and } t_n J_n \rightarrow R\}\, .
  \end{split}
  \end{equation}
  As in Section~\ref{sec:clos_graph_consmpa}, we note that if $f$ is continuous at $(x,0)$, then in the definition of $D_F(x), D_{\infty}(x)$ we no longer need to check $f_{a_n}(x_n) \rightarrow F(x)$. However, a priori, we do not impose any continuity assumption on $f$ around $\bbR^{d} \times \{ 0\}$. 

  As we establish in the next lemma, proved in Appendix~\ref{sec:rem_pfs}, due to Assumption~\ref{hyp:new_proof}, $F, D_F$ and $D_{F, \infty}$ are definable. 
  \begin{lemma}\label{lm:definab_DFinf}
    Under Assumption~\ref{hyp:new_proof}, the mappings $D_F,D_{F, \infty}$ and $F$ are definable.
  \end{lemma}

  Note that, depending on the regularity of $f$ near $(x,0)$, $D_F(x)$ can be empty. Nevertheless, our main theorem establishes that $(F,D_F)$ admits a variational stratification.
\begin{theorem}\label{thm:main}
  Under Assumption~\ref{hyp:new_proof}, for any $p\geq 1$, the pairs $(F,D_F)$, $(F, D_F + D_{F, \infty})$ and $(F, \conv(D_F + D_{F,\infty}))$ admit a $C^p$ variational stratification.
\end{theorem}

To prove Theorem~\ref{thm:main}, we actually establish a stronger statement: there is a \emph{larger} set-valued map $D_f : \bbR^{d+1} \rightrightarrows \bbR^{m\times (d+1)}$ such that $D_a(x)$ is simply the restriction of $D_f(x,a)$ onto  its first $d$ columns \emph{and} such that $(f,D_f)$ admits a definable $C^p$ variational stratification. As soon as such mapping is constructed, Theorem~\ref{thm:main} becomes a simple consequence of Theorem~\ref{thm:graph_closed_param}. To not interrupt the exposition we give the proof of this fact in Section~\ref{sec:pf_prop_proj_con}.
\begin{proposition}\label{prop:proj_cons}
  There is a definable map $D_f: \bbR^{d+1} \rightrightarrows \bbR^{m\times(d+1)}$ such that the pair $(f,D_f)$ admits a definable $C^p$ variational stratification and for every $a \neq 0$, 
  \begin{equation*}
    D_a(x) \subset \{ J \in \bbR^{m \times d} : \exists v \in \bbR^m, [J, v] \in D_f(x,a)\} \, .
  \end{equation*}
\end{proposition}

\begin{proof}[Proof of Theorem~\ref{thm:main}]
  Using Proposition~\ref{prop:proj_cons}, we note that for every $x\in \bbR^d$, $D_F(x)$ and $D_{F, \infty}(x)$ of~\eqref{eq:def_DF} and~\eqref{eq:def_DF_inf} are subsets of $\bar{D}_F(x), D_{F, \infty}(x)$ of Theorem~\ref{thm:graph_closed_param} (applied to $D_f$). Therefore, Theorem~\ref{thm:main} is just a consequence of Theorem~\ref{thm:graph_closed_param}.
\end{proof}
Clearly, if the mappings $F$ and $D_F$ are regular enough, that is $F$ is locally Lipschitz continuous and $D_F$ is a closed, nonempty valued, locally bounded map, then $D_F$ becomes a conservative mapping in the standard sense of \cite{bolte2021conservative}. A natural scenario when this is satisfied is given by the following corollary.
\begin{corollary}\label{cor:main2}
  Let Assumption~\ref{hyp:new_proof} hold, with $B= \bbR^{d} \times [0, a_0]$. Assume that for all $x \in \bbR^d$ and $a \in (0,a_0)$, $D_a(x) \neq \emptyset$ and, moreover, for any compact set $K$,
  \begin{equation*}
    \sup_{x \in K} \norm{f(x,a) - F(x)} \xrightarrow[a \rightarrow 0]{} 0
  \end{equation*}
  and 
  \begin{equation*}
    \lim_{a_0 \rightarrow 0} \sup\{ \norm{J}: J \in D_a(x)\, , \, a \leq a_0\, ,\, x \in K\} < +\infty\, .
  \end{equation*}
  Then, $D_F(x)$ can be rewritten in a simpler way:
  \begin{equation*}
    \begin{split}
      D_F(x) := \{J \in \bbR^{m \times d}: &\textrm{ there is } (x_n, a_n, J_n) \rightarrow (x,0, J)
      \, ,\textrm{ with } J_n \in D_{a_n}(x_n)\}  \, .
    \end{split}
  \end{equation*}
  And $D_F$ is conservative mapping of $F$. In particular, if $m =1$, then $D_F$ is a conservative set-valued field of $F$.
\end{corollary}

We finish this section by three examples illustrating the necessity of various points in Assumption~\ref{hyp:new_proof}.
\begin{remark}
  Observe that it is not sufficient to have $f_a$ definable for every $a>0$. Indeed, recall the example from the introduction: $f_a(x) = a \sin(x/a)$, $f'_a(x) = \sin(x/a)$, $F \equiv 0$ and $D_F \equiv [-1,1]$. If we restrict the functions to $[-1,1]$, then for all $a$, $f_a$ is definable in the structure of subanalytic sets. Nevertheless, the assumptions of the theorem do not hold since $\Graph f=\{ (x,a,y): f_a(x) = y\} $ is not definable.  
\end{remark}
\begin{remark}
  Naturally, one might wonder if an analogous version of Theorem~\ref{thm:main} might hold if $a$ is not anymore a one-dimensional parameter. Unfortunately, considering $f$ and $D$ from Remark~\ref{rm:ces_multd_field}, we see that in this case, for every $x \in (-1,1)$, $D_F(x) = \{0, 1\}$, which is not a conservative field of $F \equiv 0$.
\end{remark}

If, in the context of Corollary~\ref{cor:main2}, for each $a$, $D_a = \partial f_a$, one might wonder if $ \conv D_F$ is included in $\partial F$? Unfortunately, the following semialgebraic example, given to the author by Edouard Pauwels, shows that this is generally not true.

\begin{example}\label{ex:lim_notclarke}
  
Consider $f_a(x) = a - |x|$ if $|x| < a$ and $f_a(x) = 0$ otherwise.
A direct computation shows that $f'_{a}(a/2) = -1$ and as a consequence $-1 \in D_F(0)$. Thus, $D_F$ has elements that are different from the Clarke subgradient of $F\equiv 0$. Note that by smoothing the corners it is easy to construct a similar example where each $f_a$ is smooth.
\end{example}

\section{Applications to smoothing methods}\label{sec:smooth}

Assume that we are interested in the optimization problem
\begin{equation*}
  \min_{x \in \bbR^d} F(x)\, ,
\end{equation*} 
where $F: \bbR^d \rightarrow \bbR$ is continuous but is neither convex nor smooth. A large body of work (see for example \cite{mayne1984nondifferential,chen2012smoothing,chen1996class,garmanjani2013smoothing,hu2012smoothing,nesterov2005smooth,qi1995globally,zang1980smoothing,burke2020subdifferential,burke2017epi,zang1981discontinuous,burke2017epi,ermoliev1995minimization}) suggests tackling this problem by designing a function $f : \bbR^d \times (0, + \infty) \rightarrow \bbR$ such that for each $a >0$, $f_a$ is continuously differentiable and $f_a \xrightarrow[a \rightarrow 0]{} F$. The general optimization procedure then find, for decreasing values of $a_k, \varepsilon_k>0$, an approximately stationary point $x_k: \norm{\nabla f_{a_k}(x_k)} \leq \varepsilon_k$.

As we will see in this section, as soon as $f$ is definable, Theorem~\ref{thm:main} shows that any accumulation point of a smoothing method is a critical point of a mapping $D_F$, for which $(f,D_F)$ admits a variational stratification. Thus, it provides  theoretical guarantees for a large class of smoothing methods.

  \begin{definition}
    Let $F: \bbR^d \rightarrow \bbR$ be continuous. We say that $f: \bbR^d \times (0, + \infty) \rightarrow \bbR$ is a smoothing function of $F$ if the following holds.
    \begin{itemize}
      \item For each $x \in \bbR^d$,
      \begin{equation}\label{eq:smoothing}
        \lim_{(y,a) \rightarrow (x,0)} f(y,a) = F(x)\, .
      \end{equation}
    \item For every $a>0$, $f_a(\cdot) = f(\cdot,a)$ is $C^1$.
  \end{itemize}
  \end{definition}
  \begin{assumption}\label{hyp:smoothing}\phantom{=}
    \begin{enumerate}[i)]
      \item The functions $f: \bbR^{d}\times (0,+\infty)\rightarrow \bbR$ and $F: \bbR^d \rightarrow \bbR$ are definable.
      \item $f$ is a smoothing function for $F$.
    \end{enumerate}
   \end{assumption}
   For $x \in \bbR^d$, define
 \begin{equation*}
  D_F(x) = \{v \in\bbR^d : \textrm{ there is }(x_k, a_k, \nabla f_{a_k}(x_k)) \rightarrow (x, 0, v) \}\, .
 \end{equation*}
 Note that due to Assumption~\ref{hyp:smoothing} we don't need to check anymore if $f_{a_k}(x_k) \rightarrow F(x)$. The following proposition is a direct consequence of Theorem~\ref{thm:main}.

\begin{proposition}\label{prop:smoothing}
  Under Assumption~\ref{hyp:smoothing} the pair $(F, D_F)$ admits a definable $C^p$ variational stratification. 
\end{proposition}
As a consequence of Corollary~\ref{cor:main2}, if the smoothing family is uniformly (locally) Lipschitz, then $D_F$ is a conservative gradient of $F$.
\begin{corollary}\label{cor:smoothing_conserv}
  Let Assumption~\ref{hyp:smoothing} hold and assume that for every $x \in \bbR^d$, there is a neighborhood $\cU$ of $x$ and $C>0$ such that 
  \begin{equation*}
    \lim_{a \rightarrow 0} \sup_{x \in \cU} \norm{\nabla f_a(x)} \leq C\, .
  \end{equation*}
  Then $D_F$ is a conservative set-valued field of $F$. In particular $\partial F(x) \subset \conv D_F(x)$ and if $x^*$ is a local minimum of $F$, then 
  \begin{equation*}
    0 \in \partial F(x^*) \subset \conv D_F(x^*)\, .
  \end{equation*}
\end{corollary}

The general optimization of a smoothing method consists in finding, for decreasing values of $a_k, \varepsilon_k>0$, an approximately stationary point $x_k: \norm{\nabla f_{a_k}(x_k)} \leq \varepsilon_k$. Thus, by construction if $x^*$ is an accumulation point of a smoothing method, then $0 \in D_F(x^*)$. Therefore, if the family $(f_a)$ is uniformly Lipschitz, then the previous corollary shows that this is a necessary condition of optimality. In fact, this remains a meaningful condition even without any lipschitzness assumption.

\begin{proposition}\label{prop:smooth_locmin}
  Assume that $f: \bbR^d \times (0,+\infty)$ is a smoothing function of $F: \bbR^d \rightarrow \bbR$. If $x^*$ is a local minimum of $F$, then $0 \in D_F(x^*)$.
\end{proposition}
\begin{proof}  
  Since $x^* \in \bbR^d$ is a local minimum of $F$ there is $\delta >0$ such that for all $x \in B(x^*, \delta)$, $F(x) \geq F(x^*)$. We will show that this implies that for any $\varepsilon >0$ and $\delta' < \delta$, there is $a < \varepsilon$, and $x \in B(x^*, \delta')$ such that $\norm{\nabla f_{a}(x)} \leq \varepsilon$. This will show that there is a sequence $(x_n, a_n, \nabla f_{a_n}(x_n)) \rightarrow (x^*, 0, 0)$ and will prove the proposition.
  
  Indeed, assume the contrary, and for any $a < \varepsilon$ and $\delta' < \delta$, consider $\sx_a : \cI_a \rightarrow \bbR^d$ a (local) solution to the gradient flow equation:
    \begin{equation}\label{eq:GF}
    \dot{\sx}_a(t) = - \nabla f_a(\sx_a(t))\, ,
  \end{equation}
  starting at $x^*$, where $\cI_a$ is the maximal interval of existence. Denoting $t_a:= \inf \{ t >0: \norm{\sx_a(t) - x} \geq \delta'\}$, we obtain $t_a \in \cI_a$ and $f_a(\sx_a(t_a)) - f_a(x^*)$ is equal to
  \begin{equation*}
     -\int_{0}^{t_a}\norm{\nabla f_a(\sx(t))}^2  \leq - \varepsilon \int_{0}^{t_a} \norm{\dot{\sx}(t)}\dif t \leq - \varepsilon\norm{\sx(t_a) - x^*} = - \varepsilon \delta'\, .
  \end{equation*}
  Letting $a \rightarrow 0$, we obtain the existence of $x \in B(x^*, \delta')$ such that $F(x') < F(x)$, which is a contradiction. 
\end{proof}
\begin{remark}
  We can see from the proof of Proposition~\ref{prop:smooth_locmin}, that it is not necessary to assume $f$ definable to show that $0 \in D_F(x^*)$. Nevertheless, this inclusion might not convey any information whatsoever. For instance, approximating $x \mapsto x$ by $ x \mapsto x + a \sin(x/a)$, we obtain that $0 \in D_F(x)$, for any $x \in [-1,1]$.
\end{remark}
To summarize this section, under Assumption~\ref{hyp:smoothing}, Theorem~\ref{thm:main} establishes that to every smoothing method we can associate a mapping $D_F: \bbR^d \rightrightarrows \bbR^d$ such that $(F,D_F)$ admits a definable $C^p$ variational stratification. Moreover, $0 \in D_F(x^*)$ as soon as $x^*$ is a local minimum. Therefore, Theorem~\ref{thm:main} gives theoretical guarantees of convergence for a large class of smoothing methods.

Let us also notice, that here the non-intrinsic properties of conservative fields are becoming apparent. Indeed, $D_F$ is implicitly defined by the design of the smoothing function $f$. Depending on the latter, $D_F$ might be just the Clarke subgradient or might be larger (or smaller). Thus, the guarantees provided by Theorem~\ref{thm:main} will be dependent on the design of the smoothing function by the practitioner. We finish this section by some examples illustrating this point.

\paragraph{Finite max-functions.}
Many interesting optimization problems consider an objective $F$ that can be represented as a composition of a smooth function with a finite max-function $p : \bbR \rightarrow \bbR$, given by 
  \begin{equation}\label{eq:fin_max_fun}
    p(t) = \max_{i \leq k} a_i t + b_i\, ,
  \end{equation}
where $a_i, b_i \in\bbR$. Typical examples are $t \mapsto \max(0,t)$ or $t\mapsto |t|$.

In \cite{chen2012smoothing,burke2013gradient} the authors suggest approximating $p$ by
\begin{equation}\label{eq:smoothing_max_finite}
  s_{p,a}(t) = \int_{\bbR} p(t-au) \varrho(u) \dif u\, ,
\end{equation}
where $\varrho : \bbR \rightarrow [0,+\infty)$ is such that $\varrho(t) = \varrho(-t)$, $\int_{\bbR} \varrho(u) \dif u = 1$ and $\int_{\bbR}|u|\varrho(u) \dif u <+\infty$. In this case, by \cite[Lemma 4.1]{burke2013gradient} it holds that $s_{p,a}(\cdot)$ is a smoothing function for $p$ and, moreover, for every $t \in \bbR$, $D_p(t) = \partial p(t)$.

Consider $F(x):= g(G(H(x)))$, where $g: \bbR^m \rightarrow \bbR$ and $H: \bbR^d \rightarrow \bbR^m$ are two $C^1$ functions and $G: \bbR^m \rightarrow \bbR^m$ defined as $G(y) = [p_1(y), \ldots, p_m(y)]$, where for each $i$, $p_i$ is a finite max-function. Such $F$ naturally appears in nonlinear complementarity problems, mixed complementarity problems or regularized minimization problems (see \cite{chen2012smoothing} for more details).

Defining, $f_a(x) = g([s_{p_1, a}(H_1(x)), \ldots,s_{p_m, a}(H_m(x)) ])$, we immediately obtain that $f_a$ is a smoothing function for $F$.

\begin{proposition}\label{prop:smoothing_maxfinite}
  Assume that $g,G, H$ and $(t,a) \mapsto [s_{p_1, a}(t), \ldots, s_{p_m,a}(t)]$ are definable in the same o-minimal structure. Then, $D_F$ is a conservative set-valued field of $F$ and for every $x \in\bbR^d$,
  \begin{equation*}
    \partial F(x) \subset \conv D_F(x)\, .
  \end{equation*}
\end{proposition}
\begin{proof}
  Note that for every $a>0$, $s_{p_i,a}$ has the same Lipschitz constant as $p_i$. As a consequence, for any $a_0$ and any compact set $K \subset \bbR^d$, the family $(f_a)_{a \leq a_0}$ is uniformly Lipschitz continuous on $K$. The result follows from Corollary~\ref{cor:smoothing_conserv}.
\end{proof}\begin{remark}
  In \cite[Theorem 4.6]{burke2013gradient} it is established, without any assumptions on definability that $\conv D_F(x) = \partial F(x)$, under the condition that either $\nabla g (G(H(x)))$ has only positive coordinates or $\rank \Jac H(x)=m$. Comparing with this result we require definability of $F,G,H, s_{p,a}$, however, we do not need any regularity assumptions on $\nabla g$ and $\Jac H$. As we show in Remark~\ref{rmk:chen_counterx}, in general $\conv D_F$ might be larger than $\partial F$.
\end{remark}
\begin{remark}
  Even if $\varrho$ is definable it is not automatic that $(t,a) \mapsto s_{p,t}(a)$ is definable. We list here several cases in which this is indeed case.
  
  First, if $\varrho$ is subanalytic (for instance semialgebraic), then by \cite{lion1998integration}, the function $(t,a) \mapsto s_{p,t}(a)$ is definable in the structure log-exp. Thus, if $g,G,H$ are definable in this structure we obtain definability of $F$ and $(x,a) \mapsto f_a(x)$.

  Second, consider the case, where $p(t) = \max(0,t)$. Then, an appropriate choice of $\varrho$ produce (see~\cite{chen2012smoothing})
  \begin{equation}\label{eq:smooth_plus}
    s_{a}(t) = \begin{cases}
      \max(0,t) \quad &\textrm{ if $|t| \geq a/2$}\, ,\\
      \frac{t^2}{2a} + \frac{t}{2} + \frac{a}{8} \quad &\textrm{ otherwise}\, .
    \end{cases} \, ,  \quad \textrm{ or } \quad s_{a}(t) = \frac{1}{2} \left( t + \sqrt{t^2 + 4 a^2}\right)\, ,
  \end{equation}
  or  $s_{a}(t) = a \ln (1+ e^{t/a})$.
  Every such function is definable in the structure log-exp. Therefore, if $g,G,H$ are definable in log-exp, then Proposition~\ref{prop:smoothing_maxfinite} applies.
\end{remark}
\begin{remark}\label{rmk:chen_counterx}
  In the context of Proposition~\ref{prop:smoothing_maxfinite}, $\conv D_F$ might be larger than $\partial F$. Indeed, consider $F(t) = 2 \max(0,t) - \max(0, 2t) = 0$. Choosing the second smoothing function of~\eqref{eq:smooth_plus}, we obtain
  \begin{equation*}
    f_a(t) = \left( t + \sqrt{t^2 + 4 a^2}\right) -  \frac{1}{2} \left(2 t + \sqrt{4t^2 + 4 a^2}\right) =\sqrt{t^2 + 4 a^2} - \sqrt{t^2 + a^2} \, .
  \end{equation*}
  Thus, 
  \begin{equation*}
    f_a'(t) = \frac{t}{\sqrt{t^2 + 4 a^2}} - \frac{t}{\sqrt{t^2 + a^2}}\, .
  \end{equation*}
  Choosing $t_n = a_n$, we obtain $f_{a_n}'(t_n) = 5^{-1/2} - 2^{-1/2}$. Therefore, letting $a_n \rightarrow 0$, we obtain $5^{-1/2} - 2^{-1/2} \in D_F(0)$ while $\partial F(0) = \{ 0\}$.\footnote{In \cite[Theorem 1]{chen2012smoothing}, it was stated that if $p(t) = \max(0,t)$, then in the context of Proposition~\ref{prop:smoothing_maxfinite}, we have $\conv D_F(x) = \partial F(x)$. However, a flaw in the proof \cite[Theorem 1]{chen2012smoothing} was already noted \cite{burke2013gradient}. While in \cite{burke2013gradient} the authors conjecture that the claim of \cite[Theorem 1]{chen2012smoothing} remains true, the example of Remark~\ref{rmk:chen_counterx} shows that we need to impose some regularity condition on $\nabla g$ or $\Jac H(x)$ to obtain $\conv D_F(x) = \partial F(x)$.}
\end{remark}
\paragraph{Non Lipschitz optimization.}

Several optimization problems have an objective function that is non Lipschitz continuous. An example previously considered in literature is 
\begin{equation*}
  F(x) = \theta(x) + \sum_{i=1}^r \varphi_i(|b_i^{\top} x|^{q})\, ,
\end{equation*}
where $b_1, \ldots, b_r \in \bbR^d$, $\theta, \varphi_1, \ldots, \varphi_r: \bbR^d \rightarrow \bbR$ are $C^1$ functions and $ q >0$. Due to the absolute value, $F$ is non-smooth and if $q<1$, then $F$ is non Lipschitz continuous.

In \cite{chen2012smoothing,zhang2024riemannian,chen2013optimality,bian2012smoothing} the authors propose to use a smoothing function for $|\cdot|$:
\begin{equation}\label{eq:smooth_abs}
  s_a(t) = \begin{cases}
    |t| \quad &\textrm{ if $|s|>a$}\, ,\\
    \frac{t^2}{2 a} + \frac{a}{2} \quad &\textrm{ otherwise}\, .
  \end{cases}
\end{equation}
Then, $f_a(x) = \theta(x) + \sum_{i=1}^r \varphi_i (s_a(b_i^{\top} x)^q)$ is a smoothing function for $F$. 

Assuming, that $\theta$ and $\varphi_1, \ldots, \varphi_r$ are definable in some o-minimal structure, and, since $a,t \mapsto s_a(t)$ is semialgebraic, we obtain that $x,a \mapsto f_a(x)$ is definable. Therefore, by Proposition~\ref{prop:smoothing}, $(F,D_F)$ admits a definable $C^p$ variational stratification. We can notice that as in the case of finite max-functions $D_F(x)$ might contain elements outside $\partial F(x)$.

Finally, let us produce an example, where $D_F(x) = \emptyset$.
\begin{example}
  Consider $F(x) = \sign(x) \sqrt{|x|}$. We can smooth it by $f_a(x) = \sqrt{x + a}$ if $x>0$ and $f_a(x) = 2 \sqrt{a} - \sqrt{a - x}$ otherwise. 
  Therefore, when $x,a \rightarrow 0$, $f'_{a}(x)\rightarrow + \infty$ and $D_F(0) = \emptyset$.
\end{example}

\paragraph{Gradient consistency.} A desirable property of a smoothing method is the gradient consistency (see \cite{chen2012smoothing}): $D_F(x) \subset \partial F(x)$. Indeed, in this case $D_F$ does not produce any additional critical points. As we have seen in Example~\ref{ex:lim_notclarke}, such a property does not necessarily hold even if the family $(f_a)$ is uniformly Lipschitz continuous. However, using the fact that $(F,D_F)$ admits a variational stratification we immediately obtain gradient consistency \emph{almost everywhere}.

\section{Proof of Proposition~\ref{prop:proj_cons}}\label{sec:pf_prop_proj_con}

Let us first start with a definition of the dimension of a definable set. 

\begin{definition}[\cite{van1998tame,cos02}]
  Consider $A \subset \bbR^d$, definable, and $\bbX$ a definable $C^1$ Whitney stratification of $\bbR^d$, compatible with $A$. Then,
\begin{equation*}
  \dim(A) :=\max \{ \dim(\cX) : \cX \in \bbX\, , \cX \subset A \}\, .
\end{equation*}
\end{definition}
It can be proven that this definition is independent of the choice of the stratification $\bbX$.
We will need two simple properties of definable dimension.
\begin{lemma}[{\cite[Exercise 3.19]{cos02}}]\label{lm:dim_sections}
  Consider two definable sets $A, B \subset \bbR^{d \times 1}$, if for every $a \in \bbR$, 
  \begin{equation*}
    \dim \{ x \in \bbR^d: (x,a) \in A\} < \dim\{ x \in \bbR^d: (x,a) \in B\}\, ,
  \end{equation*}
  then $\dim(A) < \dim(B)$.
\end{lemma}
\begin{lemma}[{\cite[Proposition 3.17]{cos02}}]
  Consider definable sets $A_1, \ldots, A_k \subset \bbR^d$, it holds that $\dim(A_1 \cup \ldots A_k) = \max_{i \leq k} \dim(A_i)$.
\end{lemma}

\begin{remark}\label{rmk:stand_strat}
  A standard way of proving that there is a stratification $\bbX$, compatible with a finite number of definable sets $\bbA = \{ A_1, \ldots, A_k\}$, that satisfies some property (Whitney, Thom, etc.) is as follows. Fix any stratification $\bbX = (\cX_i)$ compatible with $\bbA$. For any $\cX \in \bbX$, consider $\bbB(\cX) \subset \cX$, the set of points on which the property (Whitney, Thom, etc.) does not hold. Prove that $\bbB(\cX)$ is definable and that $\dim \bbB(\cX) < \dim \cX$. Then, consider a stratification compatible with $\bbX$ and $\bigcup_{\cX \in \bbX} \bbB(\cX)$. Due to the decrease in the dimension, after at most $d$ re-stratifications we will obtain that $\bbB(\cX) = \emptyset$, and thus the final stratification satisfies the required property.

  In particular, this is the path taken in the proof of \cite{le1997thom}. Therefore, even if the compatibility condition was not formally stated in \cite{le1997thom}, in Proposition~\ref{prop:f_thom} we indeed can take $\bbX$ compatible with any finite collection of definable sets.
\end{remark}
The idea of the proof of Proposition~\ref{prop:proj_cons} is similar to the one exposed in Remark~\ref{rmk:stand_strat}. We construct a candidate map $G$, we look at the set of points on which conclusion of Proposition~\ref{prop:proj_cons} does not hold, and we modify the map in a way that the dimension of undesirable points decrease.

To proceed with the proof we first state two independent results. To not interrupt the exposition their proofs are gathered in Appendix~\ref{sec:rem_pfs}.
In the following, we denote $\Pi^x : \bbR^{d+1} \rightarrow \bbR^d$ the restriction onto the first $d$ coordinates.

\begin{lemma}\label{lm:trs_proj}
  Let $B_{1}, \ldots, B_k \in \bbR^{d+1}$ be definable sets. There is $\bbX = (\cX_i)$ a stratification of $\bbR^{d+1}$, compatible with $\{B_1, \ldots, B_k \}$, such that for any $a \in \bbR$ and $\cX \in \bbX$, the sets 
  $\cX_a:=\{x : (x,a) \in \cX \}$ and $\cX_a \times \{ a\}$ are $C^p$ manifolds and for any $x \in \cX_a$,
  \begin{equation}
    \cT_{\cX_a \times \{a\}}(x,a) = \cT_{\cX_a}(x) \times \{ 0\} = \cT_{\cX}(x,a) \cap \bbR^{d} \times \{ 0\}\, 
  \end{equation}
  and 
  \begin{equation*}
    \cN_{\cX_a \times \{a\}}(x,a)= \cN_{\cX_a}(x) \times \bbR  = \cN_{\cX}(x,a) + \{ 0_{\bbR^d}\} \times \bbR\, .
  \end{equation*}
\end{lemma}

\begin{lemma}\label{lm:false_strat}
  Consider $B \subset \bbR^d$,  $f: B \rightarrow \bbR^m$ and $D: \bbR^d \rightrightarrows \bbR^{m \times d}$ such that the pair $(f,D)$ admits a definable $C^p$ variational stratification. Let $\cM$ be a definable $C^p$ manifold. Then
  \begin{equation*}\label{eq:false_strats}
  \dim\{ x \in \cM : D(x) \not \subset \Jac_{\cM} f(x) + \cR_{\cM}^m(x)\} < \dim \cM\, .
  \end{equation*}
\end{lemma}

Going back to the proof of Proposition~\ref{prop:proj_cons}, we will first establish the proposition in the case where $m = 1$ (that is $f$ is a real-valued function).

For a set-valued map $G: \bbR^{d+1} \rightarrow \bbR^{d+1}$, define the set of ``bad points":
\begin{equation*}
  \bbB(G):= \{ (x,a) \in B : D_a(x)\not \subset \Pi^x(G(x,a))\}\, .
\end{equation*}
Proposition~\ref{prop:proj_cons} is a consequence of the following lemma.

\begin{lemma}\label{lm:dec_badpoints}
  Consider a definable map $G: \bbR^{d+1} \rightrightarrows \bbR^{d+1}$, for which $(f,G)$ admits a definable $C^p$ variational stratification. There is $G': \bbR^{d+1} \rightrightarrows \bbR^{d+1}$ such that the following holds. 
  \begin{enumerate}
    \item The pair $(f,G')$ admits a $C^p$ variational stratification and for every $(x,a) \not \in \bbB(G)$, $G'(x,a) = G(x,a)$. 
    \item The dimension of undesirable points decreased: $\dim \bbB(G') < \dim \bbB(G)$.
  \end{enumerate}
\end{lemma}

  \begin{proof}
    By Lemma~\ref{lm:trs_proj}, there is $\bbX = (\cX_i)$ a definable $C^p$ stratification of $\bbR^{d+1}$, compatible with $\bbB(G)$, such that $f$ is $C^p$ on each stratum \emph{and} for every $\cX \in \bbX$ and $a \in \bbR$, $\cX_a:= \{ x \in \bbR^d: (x,a) \in \cX\}$ is a submanifold, with $\cT_{\cX_a}(x) = \Pi^x(\cT_{\cX}(x,a) \cap \bbR^{d} \times \{ 0\})$.
    
    Define $G': \bbR^{d+1} \rightrightarrows \bbR^{d+1}$ as follows 
    \begin{equation*}
      G'(x,a) = \begin{cases}
        \nabla_{\cX} f(x,a) + \cN_{\cX}(x,a) \quad  &\textrm{ if $(x,a)\in \cX \subset \bbB(G)$} \\
        G(x,a) \quad &\textrm{ else.}
      \end{cases}
    \end{equation*}
    Note that for any $(x,a) \not \in \bbB(G)$, $G'(x,a) = G(x,a)$ and that $\bbB(G') \subset \bbB(G)$.
    Fix $a \in \bbR$ such that $\cX_a \neq \emptyset$. For $x \in \cX_a$, by Lemma~\ref{lm:trs_proj},
    \begin{equation*}
      \begin{split}
        \nabla_{\cX_a} f(x)+ \cN_{\cX_a}(x) &=  \Pi^{x}(\nabla_{\cX} f(x,a) + \cN_{\cX}(x,a) + \{0_{\bbR^d} \} \times \bbR)\\
        &= \Pi^{x}(\nabla_{\cX} f(x,a) + \cN_{\cX}(x,a))\\
        &= \Pi^x(G'(x,a))\, .
      \end{split}
    \end{equation*}
    
   Therefore, by Lemma~\ref{lm:false_strat}
    \begin{equation*}
     \dim\{ x \in \cX_a : (x,a) \in \bbB(G') \cap \cX\} = \dim\{ x \in \cX_a: D_a(x) \not \subset \nabla_{\cX_a} f(x) + \cN_{\cX_a}(x)\}  < \dim (\cX_a) \, .
    \end{equation*}
   Thus, by Lemma~\ref{lm:dim_sections}, $\dim(\cX \cap \bbB(G')) <  \dim( \cX) = \dim(\cX \cap \bbB(G))$. Since this holds for every $\cX \in \bbX$ such that $\cX \subset \bbB(G)$, we have established $\dim( \bbB(G')) < \dim(\bbB(G))$.
  
  \end{proof}
  
  \begin{proof}[End of the proof of Proposition~\ref{prop:proj_cons}]
    Applying Lemma~\ref{lm:dec_badpoints} at most $d+1$ times we obtain a map $G: \bbR^{d+1} \rightrightarrows \bbR^{d+1}$ such that $\bbB(G) = \emptyset$. This establishes Proposition~\ref{prop:proj_cons} in the case where $m=1$. If $m >1$, then, denoting $f = (f_1, \ldots, f_m)$, we obtain set-valued maps $G_1, \ldots, G_m : \bbR^{d+1} \rightrightarrows \bbR^{d+1}$ such that for any $(x,a)$, the $i$-th row of $D_a(x)$ is included in $\Pi^x(G_i(x,a))$. Therefore, 
    \begin{equation*}
      D_a(x) \subset [\Pi^x(G_1(x,a)), \ldots, \Pi^x(G_m(x,a)) ]^{\top}\, ,
    \end{equation*}
    which proves the proposition in the general case.
  \end{proof}

\paragraph*{Acknowledgments.} The author would like to thank Jérôme Bolte and Edouard Pauwels for interesting discussions.

\bibliographystyle{plain}
\bibliography{math}

@incollection{attouch2006convergence,
  title={Convergence de fonctionnelles convexes},
  author={Attouch, Hedy},
  booktitle={Journ{\'e}es d’Analyse Non Lin{\'e}aire: Proceedings, Besan{\c{c}}on, France, June 1977},
  pages={1--40},
  year={2006},
  publisher={Springer}
}

@article{attouch1993convergence,
  title={On the convergence of subdifferentials of convex functions},
  author={Attouch, H{\'e}dy and Beer, Gerald},
  journal={Archiv der Mathematik},
  volume={60},
  number={4},
  pages={389--400},
  year={1993},
  publisher={Birkh{\"a}user-Verlag Basel}
}

@article{poliquin1992extension,
  title={An extension of Attouch’s theorem and its application to second-order epi-differentiation of convexly composite functions},
  author={Poliquin, Ren{\'e} A},
  journal={Transactions of the American Mathematical Society},
  volume={332},
  number={2},
  pages={861--874},
  year={1992}
}

@article{levy1995partial,
  title={Partial extensions of Attouch’s theorem with applications to proto-derivatives of subgradient mappings},
  author={Levy, Adam B and Poliquin, R and Thibault, Lionel},
  journal={Transactions of the American Mathematical Society},
  volume={347},
  number={4},
  pages={1269--1294},
  year={1995}
}

@article{zolezzi1985continuity,
  title={Continuity of generalized gradients and multipliers under perturbations},
  author={Zolezzi, Tullio},
  journal={Mathematics of operations research},
  volume={10},
  number={4},
  pages={664--673},
  year={1985},
  publisher={INFORMS}
}

@article{zolezzi1994convergence,
  title={Convergence of generalized gradients},
  author={Zolezzi, Tullio},
  journal={Set-Valued Analysis},
  volume={2},
  number={1-2},
  pages={381--393},
  year={1994},
  publisher={Springer}
}

@article{czarnecki2006approximation,
  title={Approximation and regularization of Lipschitz functions: convergence of the gradients},
  author={Czarnecki, Marc-Olivier and Rifford, Ludovic},
  journal={Transactions of the American Mathematical Society},
  volume={358},
  number={10},
  pages={4467--4520},
  year={2006}
}

@article{bolte2007clarke,
  title={Clarke subgradients of stratifiable functions},
  author={Bolte, J. and Daniilidis, A. and Lewis, A. and Shiota, M.},
  journal={SIAM Journal on Optimization},
  volume={18},
  number={2},
  pages={556--572},
  year={2007},
  publisher={SIAM}
}

@article{iof08,
  title={An invitation to tame optimization},
  author={Ioffe, Alexander D},
  journal={SIAM Journal on Optimization},
  volume={19},
  number={4},
  pages={1894--1917},
  year={2009},
  publisher={SIAM}
}

@article{bolte2021conservative,
  title={Conservative set valued fields, automatic differentiation, stochastic gradient methods and deep learning},
  author={Bolte, J{\'e}r{\^o}me and Pauwels, Edouard},
  journal={Mathematical Programming},
  volume={188},
  pages={19--51},
  year={2021},
  publisher={Springer}
}

@article{pauwels2023conservative,
  title={Conservative parametric optimality and the ridge method for tame min-max problems},
  author={Pauwels, Edouard},
  journal={Set-Valued and Variational Analysis},
  volume={31},
  number={3},
  pages={1--24},
  year={2023},
  publisher={Springer}
}

@article{davis2022conservative,
  title={Conservative and semismooth derivatives are equivalent for semialgebraic maps},
  author={Davis, Damek and Drusvyatskiy, Dmitriy},
  journal={Set-Valued and Variational Analysis},
  volume={30},
  number={2},
  pages={453--463},
  year={2022},
  publisher={Springer}
}

@book{cos02,
  title={An introduction to o-minimal geometry},
  author={Coste, Michel},
  year={2000},
  publisher={Istituti editoriali e poligrafici internazionali Pisa}
}

@book{van1998tame,
  title={Tame topology and o-minimal structures},
  author={Van den Dries, Lou},
  volume={248},
  year={1998},
  publisher={Cambridge university press}
}

@article{chen2012smoothing,
  title={Smoothing methods for nonsmooth, nonconvex minimization},
  author={Chen, Xiaojun},
  journal={Mathematical programming},
  volume={134},
  pages={71--99},
  year={2012},
  publisher={Springer}
}

@article{mayne1984nondifferential,
  title={Nondifferential optimization via adaptive smoothing},
  author={Mayne, David Q and Polak, Elijah},
  journal={Journal of Optimization Theory and Applications},
  volume={43},
  number={4},
  pages={601--613},
  year={1984},
  publisher={Springer}
}

@article{gabrielov1996complements,
  title={Complements of subanalytic sets and existential formulas for analytic functions},
  author={Gabrielov, Andrei},
  journal={Inventiones mathematicae},
  volume={125},
  number={1},
  pages={1--12},
  year={1996},
  publisher={Springer}
}

@article{gabrielov1968projections,
  title={Projections of semi-analytic sets},
  author={Gabrielov, Andrei M},
  journal={Functional Analysis and its applications},
  volume={2},
  number={4},
  pages={282--291},
  year={1968},
  publisher={Kluwer Academic Publishers-Plenum Publishers New York}
}

@article{bier_semi_sub,
  title={Semianalytic and subanalytic sets},
  author={Bierstone, Edward and Milman, Pierre D},
  journal={Publications Math{\'e}matiques de l'IH{\'E}S},
  volume={67},
  pages={5--42},
  year={1988}
}

@article{wilkie1996model,
  title={Model completeness results for expansions of the ordered field of real numbers by restricted Pfaffian functions and the exponential function},
  author={Wilkie, Alex},
  journal={Journal of the American Mathematical Society},
  volume={9},
  number={4},
  pages={1051--1094},
  year={1996}
}

@article{van1994elementary,
  title={The elementary theory of restricted analytic fields with exponentiation},
  author={van den Dries, Lou and Macintyre, Angus and Marker, David},
  journal={Annals of Mathematics},
  volume={140},
  number={1},
  pages={183--205},
  year={1994},
  publisher={JSTOR}
}

@article{van96,
  title={Geometric categories and o-minimal structures},
  author={van den Dries, Lou and Miller, Chris},
  journal={Duke Math. J.},
  volume={85},
  number={1},
  pages={497--540},
  year={1996}
}

@Book{NoceWrig06,
  title={Numerical optimization},
  author={Nocedal, Jorge and Wright, Stephen J},
  year={1999},
  publisher={Springer}
}

@article{dav-dru-kak-lee-19,
  title={Stochastic subgradient method converges on tame functions},
  author={Davis, Damek and Drusvyatskiy, Dmitriy and Kakade, Sham and Lee, Jason D},
  journal={Foundations of computational mathematics},
  volume={20},
  number={1},
  pages={119--154},
  year={2020},
  publisher={Springer}
}

@article{lewis2021structure,
  title={The structure of conservative gradient fields},
  author={Lewis, Adrian S and Tian, Tonghua},
  journal={SIAM Journal on Optimization},
  volume={31},
  number={3},
  pages={2080--2083},
  year={2021},
  publisher={SIAM}
}

@article{ermoliev1995minimization,
  title={The minimization of semicontinuous functions: mollifier subgradients},
  author={Ermoliev, Yuri M and Norkin, Vladimir I and Wets, Roger JB},
  journal={Siam journal on control and optimization},
  volume={33},
  number={1},
  pages={149--167},
  year={1995},
  publisher={SIAM}
}

@book{cla-led-ste-wol-livre98,
author={Clarke, Francis H and Ledyaev, Yuri S and Stern, Ronald J and Wolenski, Peter R},
     TITLE = {Nonsmooth analysis and control theory},
    VOLUME = {178},
 PUBLISHER = {Springer-Verlag, New York},
      YEAR = {1998},
}

@incollection{tarski1951decision,
  title={A decision method for elementary algebra and geometry},
  author={Tarski, Alfred},
  booktitle={Quantifier elimination and cylindrical algebraic decomposition},
  pages={24--84},
  year={1951},
  publisher={Springer}
}

@article{chen1996class,
  title={A class of smoothing functions for nonlinear and mixed complementarity problems},
  author={Chen, Chunhui and Mangasarian, Olvi L},
  journal={Computational Optimization and Applications},
  volume={5},
  number={2},
  pages={97--138},
  year={1996},
  publisher={Springer}
}

@article{garmanjani2013smoothing,
  title={Smoothing and worst-case complexity for direct-search methods in nonsmooth optimization},
  author={Garmanjani, R and Vicente, LN},
  journal={IMA Journal of Numerical Analysis},
  volume={33},
  number={3},
  pages={1008--1028},
  year={2013},
  publisher={OUP}
}

@article{hu2012smoothing,
  title={Smoothing approach to Nash equilibrium formulations for a class of equilibrium problems with shared complementarity constraints},
  author={Hu, Ming and Fukushima, Masao},
  journal={Computational Optimization and Applications},
  volume={52},
  pages={415--437},
  year={2012},
  publisher={Springer}
}

@article{nesterov2005smooth,
  title={Smooth minimization of non-smooth functions},
  author={Nesterov, Yu},
  journal={Mathematical programming},
  volume={103},
  pages={127--152},
  year={2005},
  publisher={Springer}
}

@article{qi1995globally,
  title={A globally convergent successive approximation method for severely nonsmooth equations},
  author={Qi, Liqun and Chen, Xiaojun},
  journal={SIAM Journal on control and Optimization},
  volume={33},
  number={2},
  pages={402--418},
  year={1995},
  publisher={SIAM}
}

@article{zang1980smoothing,
  title={A smoothing-out technique for min—max optimization},
  author={Zang, Israel},
  journal={Mathematical Programming},
  volume={19},
  pages={61--77},
  year={1980},
  publisher={Springer}
}

@article{burke2020subdifferential,
  title={The subdifferential of measurable composite max integrands and smoothing approximation},
  author={Burke, James V and Chen, Xiaojun and Sun, Hailin},
  journal={Mathematical Programming},
  volume={181},
  pages={229--264},
  year={2020},
  publisher={Springer}
}

@article{burke2017epi,
  title={Epi-convergence properties of smoothing by infimal convolution},
  author={Burke, James V and Hoheisel, Tim},
  journal={Set-Valued and Variational Analysis},
  volume={25},
  pages={1--23},
  year={2017},
  publisher={Springer}
}

@article{zang1981discontinuous,
  title={Discontinuous optimization by smoothing},
  author={Zang, Israel},
  journal={Mathematics of operations research},
  volume={6},
  number={1},
  pages={140--152},
  year={1981},
  publisher={INFORMS}
}

@article{drusvyatskiy2015curves,
  title={Curves of descent},
  author={Drusvyatskiy, Dmitriy and Ioffe, Alexander D and Lewis, Adrian S},
  journal={SIAM Journal on Control and Optimization},
  volume={53},
  number={1},
  pages={114--138},
  year={2015},
  publisher={SIAM}
}

@article{bolte2023one,
  title={One-step differentiation of iterative algorithms},
  author={Bolte, J{\'e}r{\^o}me and Pauwels, Edouard and Vaiter, Samuel},
  journal={arXiv preprint arXiv:2305.13768},
  year={2023}
}

@article{bolte2024differentiating,
  title={Differentiating nonsmooth solutions to parametric monotone inclusion problems},
  author={Bolte, J{\'e}r{\^o}me and Pauwels, Edouard and Silveti-Falls, Antonio},
  journal={SIAM Journal on Optimization},
  volume={34},
  number={1},
  pages={71--97},
  year={2024},
  publisher={SIAM}
}

@article{xiao2023adam,
  title={Adam-family Methods for Nonsmooth Optimization with Convergence Guarantees},
  author={Xiao, Nachuan and Hu, Xiaoyin and Liu, Xin and Toh, Kim-Chuan},
  journal={arXiv preprint arXiv:2305.03938},
  year={2023}
}

@Book{boumal2023intromanifolds,
  title     = {An introduction to optimization on smooth manifolds},
  author    = {Boumal, Nicolas},
  publisher = {Cambridge University Press},
  year      = {2023},
  url       = {https://www.nicolasboumal.net/book},
  doi       = {10.1017/9781009166164}
}

@article{le2023nonsmooth,
  title={Nonsmooth nonconvex stochastic heavy ball},
  author={Le, Tam},
  journal={arXiv preprint arXiv:2304.13328},
  year={2023}
}

@article{le1997thom,
  title={Thom stratifications for functions definable in o-minimal structures on ({R},+,$\cdot$)},
  author={L{\^e} Loi, Ta},
  journal={Comptes Rendus de l'Academie des Sciences Series I Mathematics},
  volume={12},
  number={324},
  pages={1391--1394},
  year={1997}
}

@article{bolte2021nonsmooth,
  title={Nonsmooth implicit differentiation for machine-learning and optimization},
  author={Bolte, J{\'e}r{\^o}me and Le, Tam and Pauwels, Edouard and Silveti-Falls, Tony},
  journal={Advances in neural information processing systems},
  volume={34},
  pages={13537--13549},
  year={2021}
}

@article{burke2013gradient,
  title={Gradient consistency for integral-convolution smoothing functions},
  author={Burke, James V and Hoheisel, Tim and Kanzow, Christian},
  journal={Set-Valued and Variational Analysis},
  volume={21},
  number={2},
  pages={359--376},
  year={2013},
  publisher={Springer}
}

@inproceedings{lion1998integration,
  title={Int{\'e}gration des fonctions sous-analytiques et volumes des sous-ensembles sous-analytiques},
  author={Lion, Jean-Marie and Rolin, Jean-Philippe},
  booktitle={Annales de l'institut Fourier},
  volume={48},
  number={3},
  pages={755--767},
  year={1998}
}

@article{zhang2024riemannian,
  title={A Riemannian smoothing steepest descent method for non-Lipschitz optimization on embedded submanifolds of R n},
  author={Zhang, Chao and Chen, Xiaojun and Ma, Shiqian},
  journal={Mathematics of Operations Research},
  volume={49},
  number={3},
  pages={1710--1733},
  year={2024},
  publisher={INFORMS}
}

@article{chen2013optimality,
  title={Optimality conditions and a smoothing trust region newton method for nonlipschitz optimization},
  author={Chen, Xiaojun and Niu, Lingfeng and Yuan, Yaxiang},
  journal={SIAM Journal on Optimization},
  volume={23},
  number={3},
  pages={1528--1552},
  year={2013},
  publisher={SIAM}
}

@article{bian2012smoothing,
  title={Smoothing neural network for constrained non-Lipschitz optimization with applications},
  author={Bian, Wei and Chen, Xiaojun},
  journal={IEEE transactions on neural networks and learning systems},
  volume={23},
  number={3},
  pages={399--411},
  year={2012},
  publisher={IEEE}
}

@book{lee2022manifolds,
  title={Manifolds and differential geometry},
  author={Lee, Jeffrey M},
  volume={107},
  year={2022},
  publisher={American Mathematical Society}
}

\appendix

\section{Remaining proofs}\label{sec:rem_pfs}

\subsection{Proof of Lemma~\ref{lm:trs_proj}}
Fix $p\geq1$.
Given two $C^p$ manifolds $\cX \subset \bbR^d$ and $\cY \subset \bbR^m$, a $C^p$ function $g: \cX \rightarrow \cY$  is said to be a $C^p$ diffeomorphism if $g$ admits an inverse $g^{-1} : \cY \rightarrow \cX$ that is $C^p$. It is a submersion if for every $x \in \cX$, $\dif g(x) $ is surjective.

We recall that we denote $\Pi^x : \bbR^{d+1} \rightarrow \bbR^d$ the restriction to the first $d$ coordinates.
For $\cX \subset \bbR^{d+1}$ and $a \in \bbR$, we denote $\cX_a:= \{x  :(x,a) \in \cX \}$. Let us first establish a simple result.

\begin{lemma}\label{lm:Psimple}
  For $\cM \subset \bbR^d$ and $a \in \bbR$ assume that $\cM \times \{a\}$ is a $C^p$ submanifold of $\bbR^{d+1}$. Then, $\cM$ is a $C^p$ submanifold of $\bbR^d$ and for $x \in \cM$, 
  \begin{equation*}
      \cT_{\cM}(x)\times \{0 \} = \cT_{\cM \times \{a\}}(x,a)\, , \quad \textrm{ and } \quad \cN_{\cM}(x) \times \bbR = \cN_{\cM \times \{a\}}(x,a) \, .
  \end{equation*}
\end{lemma}
\begin{proof}
  First, $\Pi^x_{|\bbR^d \times \{a\}} : \bbR^d \times \{a\} \rightarrow \bbR^d$ is a $C^p$ diffeomorphism, with inverse $\varphi_a : x \mapsto (x,a) \in \bbR^d \times \{a\}$, and for any $x \in \bbR^d$, $\dif \varphi_a(x)[h_x] = (h_x, 0)$. Therefore, for example by \cite[Theorem 2.47]{lee2022manifolds}, $\cM = \Pi^x(\cM \times \{a\})$ is a $C^p$ submanifold of $\bbR^d$, with $
    \cT_{\cM \times \{a \}}(x,a) = \dif \varphi_a(\cT_{\cM}(x)) = \cT_{\cM}(x) \times \{ 0\}$.
It is easy then easy to see that $
    \cN_{\cM \times \{a\}}(x,a) = \cN_{\cM}(x) \times \{0\} + \{ 0_{\bbR^d}\} \times \bbR$, which completes the proof.
\end{proof}
The following lemma, combined with Lemma~\ref{lm:Psimple} establishes Lemma~\ref{lm:trs_proj}.
\begin{lemma}\label{lm:Ptrsnv_proj}
  Let $B_{1}, \ldots, B_k \in \bbR^{d+1}$ be definable sets. There is $\bbX = (\cX_i)$ a stratification of $\bbR^{d+1}$, compatible with $\{B_1, \ldots, B_k \}$, such that for any $a \in \bbR$ and $\cX \in \bbX$, the set $\cX_a \times \{ a\}$ is a $C^p$ manifold and for any $x \in \cX_a$,
  \begin{equation}\label{eq:tang_compatibility}
    \begin{split}
    \cT_{\cX_a \times \{a\}}(x,a) &= \cT_{\cX}(x,a) \cap \bbR^{d} \times \{ 0\} \, ,\\
    \cN_{\cX_a \times \{a\}}(x,a) &= \cN_{\cX}(x,a) + \{ 0_{\bbR^d}\} \times \bbR \, .     
  \end{split}
  \end{equation}
\end{lemma}
\begin{proof}
  Denote $\Pi^a: \bbR^{d+1} \rightarrow \bbR$ the definable function $\Pi^a(x,a) = a$.
  By Proposition~\ref{prop:whitney_func}, there is $\bbX = (\cX_i)$, a $C^p$ definable stratification of $\bbR^{d+1}$, compatible with $B_1, \ldots, B_k$, and $\bbA = (\cA_i)$ a $C^p$ definable stratification of $\bbR$
  such that for every $\cX \in \bbX$, there is $\cA \in \bbA$, such that $\Pi^{a}(\cX) = \cA$ and $\Pi^{a}_{|\cX}$ is of constant rank. 
  
  Since $\Pi^{a}_{|\cX} : \cX \rightarrow \cA$ is surjective, by \cite[Lemma 3.28]{lee2022manifolds} it is a submersion. By \cite[Theorem 2.41]{lee2022manifolds} this implies that for all $a \in \cA$, $(\Pi^{a}_{|\cX})^{-1}(a) = \cX_a \times \{a\} $ is a $C^{p}$ submanifold of $\bbR^{d+1}$ and by \cite[Theorem 2.47]{lee2022manifolds}
\begin{equation*}
  \cT_{\cX_a \times \{a\}}(x,a) = (\dif \Pi^{a}_{|\cX}(x,a))^{-1}(\{0\}) = \cT_{\cX}(x,a) \cap \bbR^d \times \{ 0\}  \, .
\end{equation*}  
The second equality follows since for two vector spaces $E_1, E_2$, $(E_1 \cap E_2)^{\perp} = E_1^{\perp} + E_2^{\perp}$.
\end{proof}
\subsection{Proof of Lemma~\ref{lm:false_strat}}
  Consider $\bbX$ a definable $C^p$ variational stratification of $\bbR^d$ associated with $(f,D)$ and compatible with $\cM$. Consider $\cX \in \bbX$ such that $\cX \subset \cM$. If $\dim \cX = \dim \cM$, then for any $x \in \cX$, $\cT_{\cX}(x) = \cT_{\cM}(x)$ and for all $h \in \cT_{\cX}(x)$ and $J \in D(x)$, $J h = \Jac_{\cX} f(x) h = \Jac_{\cM} f(x) h$. Therefore, the definable set on the left-hand side of Equation~\eqref{eq:false_strats} is included in strata of $\bbX$ of dimension lower than $\dim \cM$, which completes the proof.

\subsection{Proof of Lemma~\ref{lm:subg_encompas}}

If for every $a$, $D_a \equiv \partial_L f$, then 
\begin{equation*}
  \Graph D = \{(x,a,v) : \forall \varepsilon >0 \, , \exists \delta > 0\, , \norm{y-x} \leq \delta \implies f_a(y) -f_a(x) - \scalarp{v}{y-x} \geq - \varepsilon \norm{y-x} \}\, .
\end{equation*}
Therefore, in this case, $\Graph D$ is defined through a first-order formula 
(see \cite[Chapter 1]{cos02}) and is therefore definable. It is then easy to see that the set $\{(x,a,v) : v = \partial f_a(x) \}$ is also defined through a first-order formula and is therefore definable.
  \subsection{Proof of Lemma~\ref{lm:definab_DFinf}}
    We can rewrite
    \begin{equation*}
      \begin{split}
     \Graph  D_{F} = \{ (x,J) \in\bbR^{d} \times \bbR^{m \times d} : &\forall \varepsilon > 0\, , \exists (y, J_y, a) \in \bbR^{d} \times \bbR^{m \times d} \times \bbR \backslash \{ 0\} \, , J_y \in D_{a}(y)\\
      &\norm{y-x} + \norm{J_y - J} + |a| + |f(x,a) - F(x) | \leq \varepsilon  \}\,.
    \end{split}
    \end{equation*}
    Therefore, $\Graph D_F$ is defined through a first-order formula and is therefore definable. Definability of $D_{F, \infty}$ and $F$ is proven in the exact same way.

\end{document}